\documentclass{imsart}
\RequirePackage[colorlinks,citecolor=blue,urlcolor=blue]{hyperref}

\startlocaldefs
\usepackage{graphicx, verbatim}

\usepackage{amsfonts}
\usepackage{amssymb}
\usepackage{amstext}
\usepackage{amsmath}
\usepackage{amsthm,verbatim}
\usepackage{xspace}
\usepackage{graphicx}
\usepackage{algorithm}

\newtheorem{theorem}{Theorem}
\newtheorem{lemma}[theorem]{Lemma}

\newcommand{\E}{\ensuremath{\mathbb E}}
\newcommand{\R}{\ensuremath{\mathbb R}}

\newcommand{\F}{\ensuremath{\mathcal F}}

\newcommand{\lab}{\label}  \newcommand{\ra}{\ensuremath{\rightarrow}}  \def\a{{\mathbf{\alpha}}}  \def\De{{{\Delta}}}  
 \def\var{{\mathrm{var}}} \def\beq{\begin{eqnarray}} \def\eeq{\end{eqnarray}} \def\ben{\begin{enumerate}}
\def\een{\end{enumerate}} \def\bit{\begin{itemize}}
\def\bel{\begin{lemma}}
\def\eel{\end{lemma}}
\def\eit{\end{itemize}} \def\beqs{\begin{eqnarray*}} \def\eeqs{\end{eqnarray*}} \def\bel{\begin{lemma}} \def\eel{\end{lemma}}
 \newcommand{\Z}{\mathbb{Z}}  \newcommand{\C}{\mathcal{C}} 
\newcommand{\T}{\mathbb{T}}      \renewcommand{\b}{\mathbf{b}} 
\newcommand{\tb}{\tilde{\mathbf{b}}}
  \newcommand{\II}{\mathcal{I}}  \newcommand{\p}{\mathbb{P}}
   \newcommand{\e}{\mathbf{e}} 
\newcommand{\LL}{\Delta}  \newcommand{\la}{\lambda}  
   \def\eps{{\epsilon}}  \def\ie{i.\,e.\,} 
\def\vol{\mathrm{vol}}
\newcommand{\tP}{\tilde{P}}

\newcommand{\wn}{w^{(n)}}
\newcommand{\wone}{w^{(n_1)}}

\newcommand{\La}{\Lambda}

\renewcommand{\L}{\mathbb{L}} 
\renewcommand{\i}{k}
\renewcommand{\j}{\ell}
\newcommand{\ubn}{\hat{u}^{(n_1)}_b}

\newtheorem{thm}{Theorem}[section]
\newtheorem{cor}[thm]{Corollary}
\newtheorem{lem}[thm]{Lemma}

\newtheorem{claim}[thm]{Claim}
\theoremstyle{definition}
\newtheorem{defn}[thm]{Definition}
\theoremstyle{remark}

\numberwithin{equation}{section}

\endlocaldefs
\begin{document}
\begin{frontmatter}

\title{Random concave functions on an equilateral lattice with periodic Hessians I: entropy and Laplacians}
\runtitle{Random concave functions}

\author{\fnms{Hariharan} \snm{Narayanan}\ead[label=e1]{hariharan.narayanan@tifr.res.in}}
\address{
School of Technology and Computer Science, \\
Tata Institute for Fundamental Research, \\
Mumbai 400005, India.\\ \printead{e1}}

\begin{abstract}
 We show that a random concave function having a periodic hessian on an equilateral lattice has a quadratic scaling limit, if the average hessian of the function satisfies certain conditions.
We consider the set of all concave functions $g$ on an equilateral lattice $\mathbb L$ that when shifted by an element of $n \mathbb L$, incur addition by a  linear function (this condition is equivalent to the periodicity of the hessian of $g$). We identify the functions in this set, up to equality of the hessians, with a convex polytope $P_n(s)$, where $s$ corresponds to the average hessian. We show that 
the $\ell_\infty$ diameter of $P_n(s)$ is bounded below by $c(s) n^2$, where $c(s)$ is a positive constant depending only on $s$. 
Our main result is that, for any $\eps_0 > 0$, the normalized Lebesgue measure of all points in $P_n(s)$ that are not contained in a $n^2$ dimensional cube $Q$ of sidelength $2 \eps_0n^2$, centered at the unique (up to addition of a linear term) quadratic polynomial with hessian $s$, tends to $0$ as $n$ tends to $\infty$.
\end{abstract}


\end{frontmatter}

%
%
\tableofcontents

\section{Introduction}\lab{sec:intro}

\subsection{Littlewood-Richardson coefficients}
Littlewood-Richardson coefficients play an important role in the representation theory of  the general linear groups. Among other interpretations, they count the number of tilings of certain domains using squares and equilateral triangles  \cite{squaretri}. 
Let $\la, \mu, \nu$ be vectors in $\Z^n$ whose entries are non-increasing non-negative integers. Let the $\ell_1$ norm of a vector $\a \in \R^n$ be denoted $|\a|$ and let $$|\la| + |\mu| = |\nu|.$$ Take an equilateral triangle $\Delta$ of side $1$. Tessellate it with unit equilateral triangles of side $1/n$. Assign boundary values to $\Delta$ as in Figure 1; Clockwise, assign the values $0, \la_1, \la_1 + \la_2, \dots, |\la|, |\la| + \mu_1, \dots, |\la| + |\mu|.$ Then anticlockwise, on the horizontal side, assign  $$0, \nu_1, \nu_1 + \nu_2, \dots, |\nu|.$$

Knutson and Tao defined this hive model for Littlewood-Richardson coefficients in \cite{KT1}. They showed that the Littlewood-Richardson coefficient
$c_{\la\mu}^\nu$ is given by the number of ways of assigning integer values to the interior nodes of the triangle, such that the piecewise linear extension to the interior of $\Delta$ is a concave function $f$ from $\Delta$ to $\R$.  
Such  an integral ``hive" $f$ can be described as an integer point in a certain polytope known as a hive polytope. The volumes of these polytopes shed light on the asymptotics of Littlewood-Richardson coefficients \cite{Nar, Okounkov, Greta}. Additionally, 
they appear in certain calculations  in free probability \cite{KT2, Zuber}. 
\begin{figure}\label{fig:tri}
\begin{center}
\includegraphics[scale=0.25]{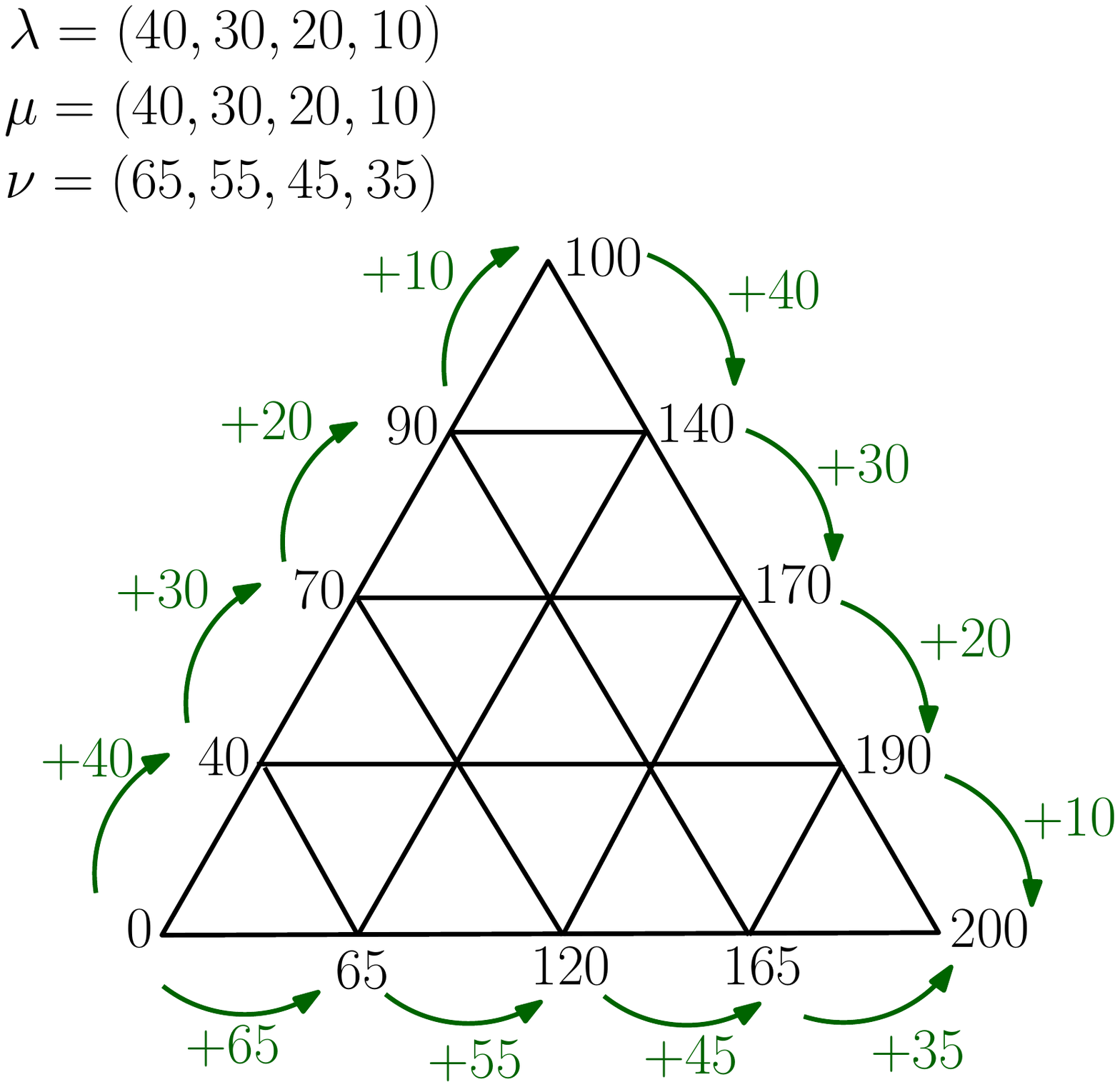}
\caption{Hive model for Littlewood-Richardson coefficients}
\end{center}
\end{figure}

\begin{figure}\label{fig:tri3}
\begin{center}
\includegraphics[scale=0.25]{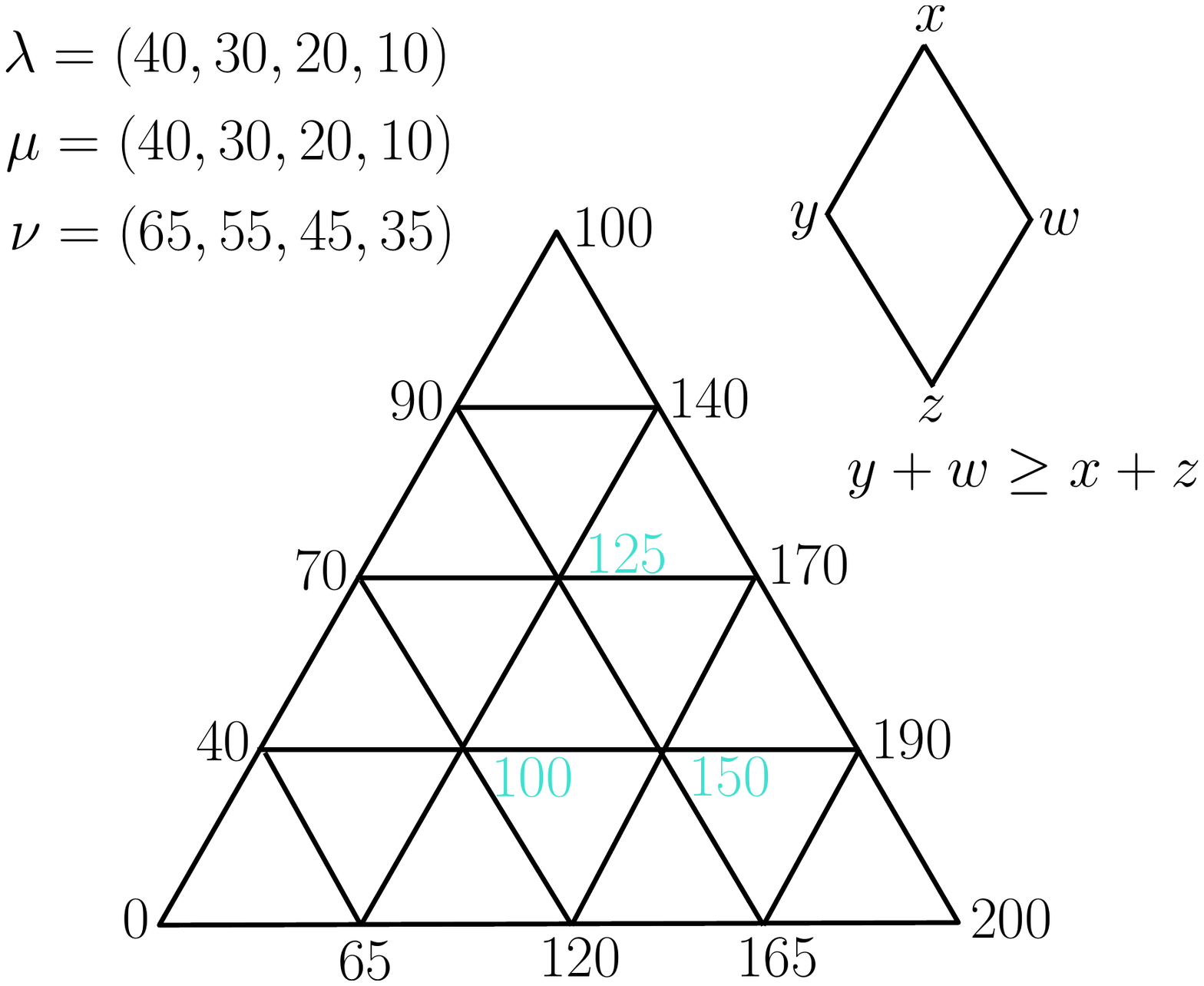}
\caption{Values taken at interior vertices in the hive model}
\end{center}
\end{figure}

The question of studying the structure of a typical real hive in a hive polytope, sampled from the Lebesgue measure is closely linked to the question of evaluating the asymptotic value of a Littlewood-Richardson coefficient for $GL_n(\mathbb C)$ as $n \ra \infty$ and $\la, \mu$ and $\nu$ tend to continuous monotonically decreasing functions in a certain fashion.
In order to study the scaling limits of random surfaces \cite{Scott}, it has proven beneficial to first examine the situation with periodic boundary conditions \cite{CohnKenyonPropp}. This is our goal in the present paper.

\subsection{Overview}
 We show that a random concave function having a periodic hessian on an equilateral lattice has a quadratic scaling limit, if the average hessian of the function satisfies certain conditions.
We consider the set of all concave functions $g$ on an equilateral lattice $\mathbb L$ that when shifted by an element of $n \mathbb L$, incur addition by a  linear function (this condition is equivalent to the periodicity of the hessian of $g$). We identify this set, up to equality of the hessians, with a convex polytope $P_n(s)$, where $s$ corresponds to the average hessian. We show in Lemma~\ref{lem:diameter} that 
the $\ell_\infty$ diameter of $P_n(s)$ is bounded below by $c(s) n^2$, where $c(s)$ is a positive constant depending only on $s$. 
We show in our main result, Theorem~\ref{thm:main} that the normalized Lebesgue measure of all points in $P_n(s)$ that are not contained in a $n^2$ dimensional cube $Q$ of sidelength $2 \eps_0 n^2$, centered at the unique (up to addition of a linear term) quadratic polynomial with hessian $s$ for any $\eps_0 > 0$, tends to $0$ as $n$ tends to $\infty$. In our proof, we construct a family $\F$ consisting of a finite number polytopes that cover $P_n(s)$, and have the following property.
For every point $x \in P_n(s)$ that is not contained in $Q$, there is a polytope $P$ in $\F$ such that $x \in P$ and $|\F|$ times the normalized volume of $P$ is bounded above by $o(1)$ as $n \ra \infty$. A large portion of this paper is devoted to establishing that the above property is true for $\F$ when $s$ belongs to a certain subset of the set of possible hessians. The key tool used for proving upper bounds on the volumes of polytopes in $\F$ is an anisotropic isoperimetric inequality (see Subsection~\ref{ssec:appl_isop}).

\subsection{Preliminaries}\lab{sec:prelim}

We consider the equilateral triangular lattice $\mathbb L$, \ie the subset of $\mathbb{C}$ generated by $1$ and $\omega = e^{\frac{2\pi \imath}{3}}$ by integer linear combinations. 
We define the edges  $E(\mathbb L )$ to be  the lattice rhombi of side $1$ in $\mathbb L$. 
We consider a rhombus  $R_n$ with vertices $0$, $n$, $n (1 - \omega^2)$ and $-n\omega^2$.  Let $\T_n$ be the torus obtained from $R_n$  by identifying opposite sides together. We define the edges  $E(\T_n)$ to be  the lattice rhombi of side $1$ in $\T_n$, where each vertex in $V(\T_n)$ is  an equivalence class of $\mathbb{L}$ modulo  $n\mathbb{L}:= n\Z   + n \omega\Z$. 

\begin{defn}[Discrete hessian]
Let $f:\mathbb{L} \ra \R$ be a function defined on $\mathbb L$.
We define the (discrete) hessian $\nabla^2(f)$ to be a  function from the set $E(\T_n)$ of rhombi of the form $\{a, b, c, d\}$ of side $1$ (where the order is anticlockwise, and the angle at $a$ is $\pi/3$) on the discrete torus to the reals, satisfying 
$$\nabla^2 f(\{a,b,c,d\}) =- f(a) + f(b) - f(c) + f(d).$$  
\end{defn}

Let $f$ be a function defined on $\mathbb{L}$ such that $\nabla^2(f)$ is periodic modulo $n\mathbb{L}$ and the piecewise linear extension of $f$ to $\mathbb C$ is concave. Such a function $f$ will be termed concave on $\L$, or simply concave. Then $\nabla^2(f)$ may be viewed as a function $g$ from $E(\T_n)$ to $\R$.

Let $a, b, c$ and $d$ be the vertices of a lattice rhombus of $ \mathbb{L}$,  of side $1$ such that  \beq\lab{eq:1.3}  a - d = -z\omega^2, \eeq \beq b-a = z,\eeq \beq c-b = -z \omega^2 ,\eeq  \beq\lab{eq:1.6}  d-c = -z,\eeq for some $z \in \{1, \omega, \omega^2\}.$ In the respective cases when $z= 1, \omega$ or $\omega^2$, we define corresponding sets of lattice rhombi of side $1$ to be $E_0(\mathbb L)$, $E_1(\mathbb L)$ or $E_2(\mathbb L)$. Note that $a$ and $c$ are vertices at which the angle is $\frac{\pi}{3}$. 
For $i = 0, 1$ and $2$, we define $E_i(\T_n)$ analogously.
For $s_0, s_1, s_2 > 0$ and $f:V(\T_n) \ra \R$, we say that $g = \nabla^2(f)$ satisfies $g \preccurlyeq s = (s_0, s_1, s_2)$, if  for all $a, b, c$ and $d$ satisfying (\ref{eq:1.3}) to (\ref{eq:1.6}) and $e = \{a, b, c, d\}$, $g$ satisfies

\ben 
\item  $g(e) \leq  s_0,$ if $e \in E_0(\T_n)$, i.e. $z = 1$.
\item $g(e)  \leq  s_1,$  if $e \in E_1(\T_n)$, i.e.  $z = \omega$.
\item $g(e) \leq  s_2,$ if  $e \in E_2(\T_n)$ i.e. $z = \omega^2$.
\een

We will further assume that $2 = s_0 \leq s_1 \leq s_2.$  Given  $s = (s_0, s_1, s_2)\in \R_+^3,$  let $P_n(s)$ be the bounded polytope of  all functions $x:V(\T_n) \ra \R$ such that $\sum_{v \in V(\T_n)} x(v) = 0$ and $\nabla^2(x)\preccurlyeq s$.

\begin{defn}Let $\tP_n(s)$ be defined to be the following image of $P_n(s)$ under an affine transformation.  Given  $s = (s_0, s_1, s_2)\in \R_+^3,$  let $\tP_n(s)$ be the bounded polytope of  all functions $x:V(\T_n) \ra \R$ such that $x(0) = 0$ and $\nabla^2(x)\preccurlyeq s$. 
\end{defn}
We observe that the $n^2-1$ dimensional Lebesgue measures of $\tP_n(s)$ and  $P_n(s)$ satisfy
$$|\tP_n(s)|^{1/n^2}\left(1 - \frac{C\log n}{n}\right) \leq |P_n(s)|^{1/n^2} \leq |\tP_n(s)|^{1/n^2}\left(1 + \frac{C\log n}{n}\right).$$

\begin{lem}\lab{lem:2.3}
For any such $s$, there is a unique quadratic function $q(s)$ from $\mathbb L$ to $\R$ such that $\nabla^2q$ satisfies the following.

\ben 
\item  $\nabla^2q(e) = - s_0,$ if $e \in E_0(\mathbb L)$.
\item $\nabla^2q(e)  =  - s_1,$  if $e \in E_1(\mathbb L )$.
\item $\nabla^2q(e)  =  - s_2,$ if  $e \in E_2(\mathbb L)$.
\item $q(0) = q(n) = q(n\omega) = 0$.
\een
\end{lem}
\begin{proof}
This can be seen by explicitly constructing $q(s)$ when $s = (1, 0, 0)$, $(0, 1, 0)$ and $(0,0,1)$ (which are rotations of the same concave function) and combining these by linear combination. Given a concave function $f:\mathbb L \ra \R$ such that 
$\nabla^2 f$ is invariant under translation by elements of $n \mathbb L$, and the average value of $\nabla^2 f$ on edges  in $E_i(\L)$ (which is well defined due to periodicity) is equal to $-s_i$ ,  and $f(0) = f(n) = f(n\omega) = 0$, we consider $(f - q)(s)$. Since the average value of $\nabla^2 f - \nabla^2 q$ is $0$, this implies that $f - q$ is $0$ on $n\L$, and more generally, is invariant under translations in $n\L$. We can therefore view $f - q$ to be a function from $\T_n = \L/{n\L}$ to $\R$, 
and in fact the resulting function is in $\tP_n(s)$. Conversely, any point in $\tP_n(s)$ can be extended to a periodic function on $\L$, to which we can add $q(s)$ and thereby recover a function $f$ on $\L$ that is concave, such that
$\nabla^2 f$ is invariant under translation by elements of $n \mathbb L$, the average value of $\nabla^2 f$ on $E_i(\L)$ is $-s_i$ ,  and $f(0) = f(n) = f(n\omega) = 0$.
\end{proof}

\underline{Note on constants:} We will use $c$ to denote an absolute positive constant that is less or equal to $1$, and $C$ to denote an absolute constant that is greater or equal to $1$. The precise values of these constants may vary from occurrence to occurrence.

\subsubsection{Convex geometry}
Let $1 \leq \ell \in \Z$. Given sets $K_i\subseteq \R^m$ for $i \in [\ell]$, let their Minkowski sum $\{x_1 + \dots + x_\ell \big| \forall i \in [\ell], x_i \in K_i\},$ be denoted by $K_1 + \dots + K_\ell.$

Let $K$ and $L$ be compact convex subsets of $\R^m$. 

Then, the Brunn-Minkowski inequality \cite{Brunn, Minkowski} states that \beq |K + L|^\frac{1}{m} \geq |K|^\frac{1}{m} + |L|^\frac{1}{m}.\eeq 
It can be shown that 
$$\lim_{\eps \ra 0^+} \frac{|L + \eps K| - |L|}{\eps}$$ exists. We will call this the anisotropic surface area $S_K(L)$ of $L$ with respect to $K$.

 Dinghas \cite{Dinghas, Figalli} showed that the following anisotropic isoperimetric inequality can be derived from the Brunn-Minkowski inequality.
\beq\lab{eq:2.2} S_K(L) \geq m|K|^{\frac{1}{m}} |L|^{\frac{m-1}{m}}.\eeq 

We shall need the following result of Pr\'{e}kopa (\cite{prekopa}, Theorem 6).
\begin{thm}\lab{thm:prekopa}
Let $f(x, y)$ be a function of $\R^n \oplus \R^m$ where $x \in \R^n$ and 
and $y \in \R^m$. Suppose that $f$ is logconcave in $\R^{n+m}$ and let
$A$ be a convex subset of $\R^m$. Then the function of the variable x:
$$\int_A f(x, y) dy$$
is logconcave in the entire space $\R^n$.
\end{thm}

We denote the $k-$dimensional Lebesgue measure of a $k-$dimensional polytope $P$ by $|P|$. We will need to show that $|P_m(s)|^{1/m^2}$ is less than $(1 + o_m(1))|P_n(s)|^{\frac{1}{n^2}},$ for $ n \geq m$. We achieve this by conditioning on a ``double layer boundary" and the use of the Brunn-Minkowski inequality.
We will identify $\Z + \Z\omega$ with $\Z^2$ by mapping $x + \omega y$, for $x, y \in \Z$ onto $(x, y) \in \Z^2.$

Given $n_1|n_2$, the natural map from $\Z^2$ to $\Z^2/(n_1 \Z^2) = \T_{n_1}$ factors through $\Z^2/(n_2 \Z^2) =\T_{n_2}$. We denote the respective resulting maps from $\T_{n_2}$ to $\T_{n_1}$ by $\phi_{n_2, n_1}$, from $\Z^2$ to $\T_{n_2}$ by $\phi_{0, n_2}$ and from $\Z^2$ to $\T_{n_1}$ by $\phi_{0, n_1}$.
Given a set of boundary nodes $\b \subseteq V(\T_n)$, and $ x \in \R^{\b}$, we define $Q_{ \b}(x)$ to be the fiber polytope over $x$, that arises from the projection map $\Pi_{\b}$ of $\tP_n(s)$ onto $\R^{\b}.$ Note that $Q_{\b}(x)$ implicitly depends on $s$.

Let $\{0\} \subseteq \b_1 \neq \{0\},$ be a subset of $V(\T_{n_1}).$ 
 Given any vertex $v_1$ in $\b_1$ other than $0$, there is a lattice path $\mathrm{path}(v_1)$ (i.e. a path $0= a_1, \dots, a_k = v_0$, where each $a_i - a_{i-1}$ is in the set $\{1, 1 + \omega, \omega, -1, \omega^2, 1 - \omega^2\}$) that goes from $0$ to some vertex $v_0 \in \phi^{-1}_{0,n_1}(v_1)$ that consists of two straight line segments, the first being from $0$ to some point in $\Z^+$, and the second having the direction $1 + \omega$. It is clear that this $v_0$ can be chosen to have absolute value at most $2n_1$ by taking an appropriate representative of $\phi^{-1}_{0,n_1}(v_1).$
We see that $[0,1]^{\b_1\setminus\{0\}} \subseteq \Pi_{\b_1} \tP_{n-1}(s) \subseteq \R^{\b_1\setminus \{0\}}.$ Let $f_1 \in  \tP_{n_1}(s)$. Along $\mathrm{path}(v_1)$, at each step, the slope of $f$ increases by no more than a constant, due to the condition $\nabla^2(f_1) \preccurlyeq s.$ This implies that $f_1$ is $Cn_1$ Lipschitz. Therefore, $\|f_1\|_{\ell_\infty}$ is at most $Cn_1^2.$ Thus $\Pi_{\b_1} \tP_{n_1}(s)$ is contained inside a $|\b_1| -1$ dimensional cube of side length no more than $Cn_1^2.$ 
We have thus proved the following.

\begin{lem}\lab{lem:3-}
Let $\{0\} \subseteq \b_1 \neq \{0\},$ be a subset of $V(\T_{n_1}).$ Then,
$$0 \leq \ln |\Pi_{\b_1} \tP_{n_1}(s)| \leq (|\b_1| -1)\ln (Cn_1^2).$$
\end{lem}

\section{Characteristics of relevant polytopes}\lab{sec:2}

\subsection{Volume of the polytope $P_n(s)$}

\begin{lem}\lab{lem:3}
Let $n_1$ and $n_2$ be positive integers satisfying $n_1 | n_2$. Then 
\beq 1 \leq |\tP_{n_1}(s)|^{\frac{1}{n_1^2}} \leq |\tP_{n_2}(s)|^{\frac{1}{n_2^2}}\left(1 + \frac{C\log n_1}{n_1}\right).\eeq 
\end{lem}
\begin{proof}The lower bound of $1$ on $|\tP_{n_1}(s)|^{\frac{1}{n_1^2}}$ follows from $[0,1]^{V(\T_{n_1})\setminus\{0\}} \subseteq \tP_n(s).$
We define the set $\b_{1} \subseteq V(\T_{n_1})$ of ``boundary vertices" to be all vertices that are either of the form $(0, y)$ or $(1, y)$ or $(x, 0)$ or $(x, 1)$, where $x, y$ range over all of $\Z/(n_1 \Z)$. We define the set $\b_{2}$ to be $\phi_{n_2, n_1}^{-1}(\b_{1}).$ For $x \in \R^{\b_1}$, let $F_1(x):= |Q_{\b_1}(x)|,$ and for  $x \in \R^{\b_2}$, let $F_2(x):= |Q_{\b_2}(x)|.$
We now have \beq|\tP_{n_1}(s)| = \int\limits_{\R^{\b_1}}F_1(x)dx = \int\limits_{\Pi_{\b_1} \tP_{n_1}(s)} F_1(x) dx.\eeq
Let  $\phi^*_{n_2, n_1}$ be the linear map from $\R^{V(\T_{n_1})}$ to $\R^{V(\T_{n_2})}$ induced by $\phi_{n_2, n_1}$. 
Let  $\psi_{\b_1, \b_2}$ be the linear map from $\R^{\b_1}$ to $\R^{\b_2}$ induced by $\phi_{n_2, n_1}$. 
Then, for $x \in \R^{\b_1},$ \beq F_2(\psi_{\b_1, \b_2}(x)) = F_1(x)^{\left(\frac{n_2}{n_1}\right)^2}.\lab{eq:2.6}\eeq
Note that that $\tP_{n}(s)$ is $n^2-1$ dimensional, has an $\ell_\infty$ diameter of $O(n^2)$ and contains a $n^2-1$ dimensional unit $\ell_\infty-$ball as a consequence of $s_0$ being set to  $2$. So the $|\b_1|-1$ dimensional polytopes  $\Pi_{\b_1} \tP_{n_1}(s)$, and $\psi_{\b_1, \b_2}(\Pi_{\b_1} \tP_{n_1}(s))$ contain $|\b_1|-1$ dimensional $\ell_\infty$ balls of radius $1.$ 
\begin{claim} \lab{cl:2.2}
Let $S_{\b_1, \b_2}(\frac{1}{n_1^{4}})$ be the set of all $y \in \R^{\b_2}$ such that there exists $x \in  \Pi_{\b_1} \tP_{n_1}((1 - \frac{1}{n_1^2})s)$ for which  $ y - \psi_{\b_1, \b_2}(x) \perp \psi_{\b_1, \b_2}(\R^{\b_1})$  and $\|y - \psi_{\b_1, \b_2}(x) \|_{\ell_\infty} < \frac{1}{n_1^{4}}.$ Then, $y \in S_{\b_1, \b_2}(\frac{1}{n_1^{4}})$ implies the following. \ben \item $y \in \Pi_{\b_2} \tP_{n_2}((1 - \frac{1}{2n_1^2})s)$ and \item 
 $|Q_{\b_2} (y)| \geq c^{(\frac{n_2}{n_1})^2} |Q_{\b_2} (\psi_{\b_1, \b_2}(x))|.$\een
\end{claim}
\begin{proof}
The first assertion of the claim follows from the triangle inequality. To see the second assertion, 
let the vector $w \in \R^{V(\T_{n_2})}$ equal $0$ on all the coordinates indexed by $V(\T_{n_2})\setminus \b_2$ and equal $\psi_{\b_1, \b_2}(x) - y$ on coordinates indexed by $\b_2$.
We know that $x \in  \Pi_{\b_1} \tP_{n_1}((1 - \frac{1}{n_1^2})s)$. Therefore, \ben \item[$(\ast)$] $Q_{\b_2} (\psi_{\b_1, \b_2}(x)) -w$ has dimension $n_2^2 - |\b_2|$, and contains an axis aligned cube of side length $\frac{c}{n_1^2},$ and hence a euclidean ball of radius $\frac{c}{n_1^2}$.\een Since every constraint defining $\tP_{n_2}(s)$ has the form $x_a + x_b - x_c - x_d \leq s_i,$ or $x_0 = 0$, \ben \item[$(\ast \ast)$] the affine spans of the codimension $1$ faces of the fiber polytope $Q_{\b_2}(y)$ are respectively translates of the affine spans of the corresponding codimension $1$ faces of $Q_{\b_2} (\psi_{\b_1, \b_2}(x)) -w$ by  euclidean distances that do not exceed $\frac{C}{ n_1^{4}}.$ \een
Therefore, by $(\ast)$ and $(\ast \ast)$, some translate of $(1 - \frac{C}{n_1^{2}})Q_{\b_2} (\psi_{\b_1, \b_2}(x))$ is contained inside $Q_{\b_2} (y)$, completing the proof of Claim~\ref{cl:2.2}.
\end{proof}
Let $K$ denote the intersection of the origin symmetric cube of radius $\frac{1}{n_1^{4}}$ in $\R^{\b_2}$ with the orthocomplement of $\psi_{\b_1, \b_2}(\R^{\b_1})$. By the lower bound of $1$ on the volume of a central section of the unit  cube (due to Vaaler \cite{Vaaler}), it follows that the volume of $K$ is at least $\left(\frac{1}{n_1^{4}}\right)^{|\b_2| - |\b_1|}.$
The inequalities below now follow from (\ref{eq:2.6}) and Claim~\ref{cl:2.2}.
\beqs |\tP_{n_2}(s)|  & = & \int\limits_{\Pi_{\b_2} \tP_{n_2}(s)}|Q_{\b_2}(y)|dy\\
                             & \geq & \int\limits_{\Pi_{\b_2} \tP_{n_2}((1 - \frac{1}{2n_1^2})s)} F_2(y) dy\lab{eq:2.8}\\
& \geq & \int\limits_{ S_{\b_1, \b_2}(\frac{1}{n_1^4})} F_2(y) dy\\
& \geq & \vol(K)\int\limits_{\psi_{\b_1, \b_2}(\Pi_{\b_1} \tP_{n_1}((1 - \frac{1}{n_1^2})s))} c^{(\frac{n_2}{n_1})^2}F_2(z) dz\\
& \geq &  \vol(K)\int\limits_{\Pi_{\b_1} \tP_{n_1}((1 - \frac{1}{n_1^2})s)}c^{(\frac{n_2}{n_1})^2} F_1(x)^{\left(\frac{n_2}{n_1}\right)^2}dx\nonumber\\
& \geq &  c^{(\frac{n_2}{n_1})^2}\left(\frac{1}{n_1^{4}}\right)^{|\b_2|-|\b_1|}\int\limits_{\Pi_{\b_1} \tP_{n_1}((1 - \frac{1}{n_1^2})s)} F_1(x)^{\left(\frac{n_2}{n_1}\right)^2}dx.\eeqs

By Lemma~\ref{lem:3-}, $ n_1^{-Cn_1} \leq |\Pi_{\b_1} \tP_{n_1}(s)| \leq n_1^{Cn_1}$, for some universal positive constant $C > 1.$ Also, $c  |\Pi_{\b_1} \tP_{n_1}(s)| \leq |\Pi_{\b_1} \tP_{n_1}((1 - \frac{1}{n_1^2})s)| \leq |\Pi_{\b_1} \tP_{n_1}(s)|.$
\beqs \int\limits_{\Pi_{\b_1} \tP_{n_1}((1 - \frac{1}{n_1^2})s)} F_1(x)^{\left(\frac{n_2}{n_1}\right)^2}dx & \geq & |\Pi_{\b_1} \tP_{n_1}((1 - \frac{1}{n_1^2})s)|^{1 - (n_2/n_1)^2}  \\ &\times&\left(\int\limits_{\Pi_{\b_1} \tP_{n_1}((1 - \frac{1}{n_1^2})s)} F_1(x)dx\right)^{\left(\frac{n_2}{n_1}\right)^2}\\
& \geq &  |\Pi_{\b_1} \tP_{n_1}(s)|^{1 - (n_2/n_1)^2}  | \tP_{n_1}((1 - \frac{1}{n_1^2})s)|^{\left(\frac{n_2}{n_1}\right)^2}\\
& \geq &  |\Pi_{\b_1} \tP_{n_1}(s)|^{1 - (n_2/n_1)^2}  \left(c| \tP_{n_1}(s)|\right)^{\left(\frac{n_2}{n_1}\right)^2}\\
& \geq & (Cn_1^{Cn_1})^{1 - (n_2/n_1)^2} | \tP_{n_1}(s)|^{\left(\frac{n_2}{n_1}\right)^2}. \eeqs
Thus,
\beq | \tP_{n_1}(s)|^{\left(\frac{n_2}{n_1}\right)^2} \leq (Cn_1^{Cn_1})^{(n_2/n_1)^2-1} \left(n_1^{4}\right)^{|\b_2|-|\b_1|}|\tP_{n_2}(s)|,  \eeq
which gives us \beq| \tP_{n_1}(s)|^{\left(\frac{1}{n_1}\right)^2} & \leq & (Cn_1^{Cn_1})^{(1/n_1^2)-(1/n_2^2)} \left(n_1^{4}\right)^{\frac{|\b_2|-|\b_1|}{n_2^2}}|\tP_{n_2}(s)|^{\frac{1}{n_2^2}}\\
& \leq & |\tP_{n_2}(s)|^{\frac{1}{n_2^2}}n_1^{\frac{C}{n_1}}\\
& \leq & |\tP_{n_2}(s)|^{\frac{1}{n_2^2}}\left(1 + \frac{C \log n_1}{n_1}\right).\eeq
\end{proof}
For a positive integer $n$, let $[n]$ denote the set of positive integers less or equal to $n$, and let $[n]^2$ denote $[n]\times [n]$.
In what follows, we will use $v$ to denote an arbitrary vertex in $V(\T_{n_3})$. Then, by symmetry, \beq \frac{\int_{P_{n_3}(s)} x(v) dx}{|P_{n_3}(s)|} & = &  \left(\frac{1}{n_3^2}\right)\sum_{v' \in V(\T_{n_3})}  \frac{\int_{P_{n_3}(s)} x(v') dx}{ |P_{n_3}(s)|} \\
& = &   \frac{\int_{P_{n_3}(s)} \left(\frac{\sum_{v' \in V(\T_{n_3})} x(v')}{n_3^2}\right)dx}{|P_{n_3}(s)|} \\
& = & 0.\lab{eq:2.10} \eeq
The linear map $u:P_{n_3}(s) \rightarrow \tP_{n_3}(s)$ defined  by $u(x)(v) = x(v) - x(0)$ is surjective and volume preserving. Therefore,  
\beq \frac{\int_{\tP_{n_3}(s)} x(v) dx}{|\tP_{n_3}(s)|} & = &   \frac{\int_{P_{n_3}(s)} u(x)(v) dx}{ |P_{n_3}(s)|} \\
& = &  \frac{\int_{P_{n_3}(s)} x(v) dx}{|P_{n_3}(s)|} -  \frac{\int_{P_{n_3}(s)} x(0) dx}{|P_{n_3}(s)|}\\
& = & 0.\eeq
\begin{lem}\lab{lem:4}
Let $C < n_2 < n_3$. Then, 
\beq |P_{n_2}(s)|^{\frac{1}{n_2^2}} \geq |P_{n_3}(s)|^{\frac{1}{n_3^2}}\left(1 - \frac{C (n_3 - n_2) \ln n_3}{n_3}\right).\eeq  
\end{lem}
\begin{proof}  Let $\rho:V(\T_{n_2}) \ra [n_2]^2\subseteq \Z^2$ be the unique map that satisfies $\phi_{0, n_2} \circ \rho = id$ on $V(\T_{n_2})$. We embed $V(\T_{n_2})$ into $V(\T_{n_3})$ via  
  $\phi_{0, n_3} \circ \rho,$ and define $\b$ to be $ V(\T_{n_3})\setminus (\phi_{0, n_3} \circ \rho(V(\T_{n_2}))).$  Note that $0 \in \b$, since $0 \not \in [n_2].$ Recall that $Q_{ \b}(x)$ was defined to be the fiber polytope over $x$, that arises from the projection map $\Pi_{\b}$ of $\tP_n(s)$ onto $\R^{\b}.$
Thus,
\beqs \int\limits_{\R^{\b\setminus \{0\}}} \left(\frac{|Q_{\b}(x)| }{|\tP_{n_3}(s)|}\right)x dx & = & \Pi_\b  \left(\frac{\int_{\tP_{n_3}(s)} x(v) dx}{|\tP_{n_3}(s)|}\right)\\ 
& = & 0.\eeqs
By Theorem~\ref{thm:prekopa}, $\frac{|Q_{\b}(x)| }{|\tP_{n_3}(s)|}$ is a logconcave function of $x\in \tP_{n_3}(s)$. 
$\frac{|Q_{\b}(x)| }{|\tP_{n_3}(s)|}$ is a non-negative and integrable function of $x$, and hence by the Brunn-Minkowski inequality, it follows that
\beqs \int\limits_{\R^{\b\setminus \{0\}}} \left(\frac{|Q_{\b}(x)|}{|\tP_{n_3}(s)|}\right)|Q_{\b}(x)|^{\frac{1}{n_3^2 - |\b|}} dx \leq |Q_{\b}(0)|^{\frac{1}{n_3^2 - |\b|}} .\eeqs
Therefore, 
\beqs \int\limits_{\Pi_{\b} \tP_{n_3}(s)} |Q_{\b}(x)|^{1 + \frac{1}{n_3^2 - |\b|}} \left(\frac{dx}{|\Pi_\b\tP_{n_3}(s)|}\right) \leq \left(\frac{| \tP_{n_3}(s)|}{|\Pi_{\b} \tP_{n_3}(s)|}\right)|Q_{\b}(0)|^{\frac{1}{n_3^2 - |\b|}}.\eeqs
By the monotonic increase of $L_p(\mu)$ norms as $p$ increases from $1$ to $\infty$,  for the probability measure $\mu(dx) = \frac{dx}{|\Pi_\b\tP_{n_3}(s)|}$, we see that 
\beq \int\limits_{\Pi_{\b}\tP_{n_3}(s)} |Q_{\b}(x)|^{1 + \frac{1}{n_3^2 - |\b|}} \frac{dx}{|\Pi_\b \tP_{n_3}(s)|} & \geq &  \left(\int\limits_{\Pi_{\b}\tP_{n_3}(s)} |Q_{\b}(x)| \frac{dx}{|\Pi_\b \tP_{n_3}(s)|}\right)^{1 + \frac{1}{n_3^2 - |\b|}}\\ & = & \left(\frac{| \tP_{n_3}(s)|}{|\Pi_{\b} \tP_{n_3}(s)|}\right)^{1 + \frac{1}{n_3^2 - |\b|}}.\eeq
It follows that \beq |Q_{\b}(0)| \geq \frac{| \tP_{n_3}(s)|}{|\Pi_{\b} \tP_{n_3}(s)|}.\eeq
Suppose that $n_2 + 2 < n_3$.
Let $\rho_+:V(\T_{n_2+2}) \ra [n_2+2]^2\subseteq \Z^2$ be the unique map that satisfies $\phi_{0, n_2+2} \circ \rho_+ = id$ on $V(\T_{n_2+2})$. We embed $V(\T_{n_2+2})$ into $V(\T_{n_3})$ via  
  $\phi_{0, n_3} \circ \rho_+,$ and define $\tb$ to be $ V(\T_{n_3})\setminus (\phi_{0, n_3} \circ \rho_+(V(\T_{n_2+2}))).$ 
We observe that $|\tP_{n_2+2}(s(1 + \frac{2}{(n_2+2)^2}))|$ is greater or equal to $|Q_{\b}(0)|(\frac{1}{(n_2+2)^2}))^{|\b|-|\tb|},$ since $\phi_{0, n_3} \circ \rho_+,$ induces an isometric map from $Q_\b(0) + [0, \frac{1}{(n_2+2)^2}]^{\b\setminus\tb}$ into $\tP_{n_2+2}(s(1 + \frac{2}{(n_2+2)^2})).$
Thus,
\beqs |\tP_{n_2 + 2}(s)| & = &  (1 + \frac{2}{(n_2+2)^2})^{-(n_2+2)^2+1}|\tP_{n_2+2}(s(1 + \frac{2}{(n_2+2)^2}))|\\
&  \geq & e^{-2} |Q_{\b}(0)|(\frac{1}{(n_2+2)^2})^{|\b|-|\tb|}\\
&  \geq & \frac{e^{-2}| \tP_{n_3}(s)|(\frac{1}{(n_2+2)^2})^{|\b|-|\tb|} }{|\Pi_{\b} \tP_{n_3}(s)|}\\
& \geq & | \tP_{n_3}(s)| (Cn_3)^{-Cn_3(n_3-n_2)}.\eeqs
Noting that $\tP_{n_2+2}(s)$ contains a unit cube and hence has volume at least $1$, we see that 
\beq |\tP_{n_2 + 2}(s)|^{\frac{1}{(n_2 + 2)^2}} & \geq & |\tP_{n_2+2}(s)|^{\frac{1}{n_3^2}}\\
& \geq & | \tP_{n_3}(s)|^{\frac{1}{n_3^2}} (C n_3)^{-C(1-\frac{n_2}{n_3})}\\
& \geq &  | \tP_{n_3}(s)|^{\frac{1}{n_3^2}} \left(1 - \frac{C (n_3 - n_2) \ln n_3}{n_3}\right).\eeq
Noting that $n_2 + 2 < n_3$ and relabeling $n_2 + 2$ by $n_2$  gives us the lemma.
\end{proof}

We will need the notion of differential entropy (see page 243 of \cite{Cover}).
\begin{defn}[Differential entropy]
Let ${\displaystyle X}$ be a random variable supported on a finite dimensional Euclidean space $\R^m$,
associated with a measure $\mu$ that is absolutely continuous with respect to the Lebesgue measure. Let the Radon-Nikodym derivative of $\mu$ with respect to the Lebesgue measure be denoted $f$. The differential entropy of $X$, denoted ${\displaystyle h(X)}$ (which by overload of notation, we shall also refer to as the differential entropy of $f$, i.e. $h(f)$), is defined  as
${\displaystyle h(X)=-\int _{\R^m}f(x)\ln f(x)\,dx}$.
\end{defn}

We will also need the notion of conditional differential entropy ${\displaystyle h(X|Y)}$ (page 249 of \cite{Cover}).

\begin{defn}[Conditional differential entropy]

Let ${\displaystyle (X, Y)}$ where $X \in \R^m$, and $Y \in \R^n$ be a random variable supported on a finite dimensional Euclidean space $\R^{m}\times \R^{n}$ having a joint density function $f(x, y)$ with respect to the Lebesgue measure on $\R^{m}\times \R^{n}$. Then,
$${\displaystyle h(X|Y)=-\int _{\R^{m} \times {\R^n}} f(x,y)\log f(x|y)\,dxdy}.$$

 Let the Radon-Nikodym derivative of $\mu$ with respect to the Lebesgue measure be denoted $f$. The differential entropy of $X$, denoted ${\displaystyle h(X)}$ (which by overload of notation, we shall also refer to as the differential entropy of $f$, i.e. $h(f)$), is defined  as
${\displaystyle h(X)=-\int _{\R^m}f(x)\ln f(x)\,dx}$.
\end{defn}
The following Lemma is well known, but we include a proof for the reader's convenience.
\begin{lem}\lab{lem:5}
The differential entropy of a  mean $1$ distribution with a bounded Radon-Nikodym derivative with respect to the Lebesgue measure, supported on $[0, \infty)$ is less or equal to $1$, and equality is achieved on the exponential distribution.
\end{lem}
\begin{proof}
Let $f:[0, \infty) \ra \R$ denote a density supported on the non-negative reals, whose associated distribution $F$ has mean $1$. Let $g:[0, \infty) \ra \R$ be given by $g(x) := e^{-x}$. The relative entropy between $f$ and $g$ is given by 

\beq D(f||g) := \int_{[0, \infty)} f(x) \ln\left(\frac{f(x)}{g(x)}\right)dx,\eeq and can be shown to be non-negative for all densities $f$ using Jensen's inequality.
We observe that \beq D(f||g) & = & -h(f) +  \int_{[0, \infty)} f(x) \ln\left({e^{x}}\right)dx\\
                                          & = & - h(f) + 1, \eeq because $F$ has mean $1$.
This implies that $h(f) \leq 1 = h(g)$.
\end{proof}

\begin{lem}\lab{lem:6}
If $s_0= 2$, \beqs |P_{n+1}(s)| \leq \exp\left(Cn \ln Cn + (1 + \ln 2) n^2\right),\eeqs
\end{lem}
\begin{proof}
  Suppose without loss of generality that $$\phi_{0, n + 1}\left(\{(1, 1), (1, 2), (2, 1), (2, 2)\}\right) \in E_0(\T_{n+1}),$$ and thus that the hessian corresponding to this edge is at most $s_0 = 2$. Define $\b^{(0)}$ to be $ V(\T_{n+1})\setminus (\phi_{0, n + 1}([n]^2).$ We list $[n]^2$ in lexicographically increasing order as $$((1, 1), \dots, (1, n), (2, 1), \dots, (2, n), \dots, (n, 1), \dots, (n, n))$$  and denote this sequence by $((p_1, q_1), \dots, (p_{n^2}, q_{n^2}))$, and iteratively define $\b^{(i)} := \b^{(i-1)}\cup\{\phi_{0, n + 1} ((p_i, q_i))\}$ for $i = 1$ to $n^2$. We see that for each $i \in [n^2]$, $\phi_{0, n + 1} ((p_i-1, q_i))$ and $\phi_{0, n + 1} ((p_i-1, q_i-1))$ and $\phi_{0, n + 1} ((p_i, q_i-1))$ are in $\b^{(i-1)}$. Let these respectively correspond to  $\b^{(E(i))}\setminus \b^{(E(i)-1)} and  \b^{(SE(i))}\setminus \b^{(SE(i)-1)}$ and $\b^{(S(i))}\setminus \b^{(S(i)-1)}.$ Let $x$ be a random sample from the uniform probability measure on $P_{n+1}(s).$ and let $x(\phi_{0, n+1}((p_i-1, q_i))) =x_i.$
We see that \beqs h\left(\Pi_{\b^{(i)}} x \big | \Pi_{\b^{(i-1)}} x\right) & = & h\left(\Pi_{\b^{(i)}} x - \Pi_{\b^{(i-1)}} x \big | \Pi_{\b^{(i-1)}} x\right)\\
& = & h\left(x_i \big |  \Pi_{\b^{(i-1)}} x\right)\\
& = & h\left(x_i - \left(x_{E(i)} + x_{S(i)} - x_{SE(i)}\right) \big |  \Pi_{\b^{(i-1)}} x\right)\\
& \leq & h\left(x_i - x_{E(i)} - x_{S(i)} + x_{SE(i)}\right).\eeqs
The random variable $x_i - x_{E(i)} - x_{S(i)} + x_{SE(i)}$ is bounded above by $s_0 = 2$, and by 
(\ref{eq:2.10}) has expectation $0$.
 The conditions of Lemma~\ref{lem:5} are thus satisfied by $\frac{2 - x_i + x_{E(i)} + x_{S(i)} - x_{SE(i)}}{2}$, giving us 
\beq h\left(x_i - x_{E(i)} - x_{S(i)} + x_{SE(i)}\right) \leq 1 + \ln 2. \lab{eq:ln2}\eeq
Since \beq h(x) \leq h(\Pi_{\b^{(0)}} x) + \sum_{i \in [n]^2} h(\Pi_{\b^{(i)}} x\big | \Pi_{\b^{(i-1)}} x ),\eeq
by Lemma~\ref{lem:3} and (\ref{eq:ln2}), we have 
\beq h(x) \leq Cn \ln Cn + (1 + \ln 2) n^2. \eeq
This implies that for all positive integers $n$, \beq|P_{n+1}| \leq \exp\left(Cn \ln Cn + (1 + \ln 2) n^2\right),\eeq
implying in particular that $|P_n(s)|^{\frac{1}{n^2}}$ is bounded above by  $C$.
\end{proof}

We will use the lemmas in this section to prove the following.

\begin{lem}\lab{lem:7}
Let $s_0 = 2$. Then, as $n\ra \infty$, $|P_n(s)|^{\frac{1}{n^2}}$  converges to a limit in the interval $[1, 2e]$.
\end{lem}
\begin{proof}
By Lemma~\ref{lem:3} and Lemma~\ref{lem:6}, \beq 1 \leq  \liminf\limits_{n \ra \infty} |P_n(s)|^{\frac{1}{n^2}} \leq \limsup\limits_{n \ra \infty} |P_n(s)|^{\frac{1}{n^2}} \leq 2e.\eeq

Let $C < n_1^2  \leq n_2.$ Let $n_3 = (\lfloor \frac{n_2}{n_1}\rfloor + 1)  n_1.$ 
By Lemma~\ref{lem:3} and Lemma~\ref{lem:4},
\beqs |P_{n_1}(s)|^{\frac{1}{n_1^2}} & \leq &  |P_{n_3}(s)|^{\frac{1}{n_3^2}}\left(1 + \frac{C\log n_1}{n_1}\right)\\
& \leq & |P_{n_2}(s)|^{\frac{1}{n_2^2}}\left(1 - \frac{C (n_3 - n_2) \ln n_3}{n_3}\right)^{-1}\left(1 + \frac{C\log n_1}{n_1}\right)\\
& \leq & |P_{n_2}(s)|^{\frac{1}{n_2^2}}\left(1 - \frac{C n_1 \ln n_3}{n_3}\right)^{-1}\left(1 + \frac{C\log n_1}{n_1}\right)\\ 
& \leq & |P_{n_2}(s)|^{\frac{1}{n_2^2}}\left(1 - Cn_1 \left(\frac{\ln n_1^2}{n_1^2}\right)\right)^{-1}\left(1 + \frac{C\log n_1}{n_1}\right).\eeqs
This implies that \beqs |P_{n_2}(s)|^{\frac{1}{n_2^2}} \geq |P_{n_1}(s)|^{\frac{1}{n_1^2}}\left(1 - \frac{C\log n_1}{n_1}\right).\eeqs
As a consequence, \beqs \left(1 + \frac{C\log n_1}{n_1}\right)\liminf\limits_{n_2 \ra \infty} |P_{n_2}(s)|^{\frac{1}{n_2^2}} \geq |P_{n_1}(s)|^{\frac{1}{n_1^2}}. \eeqs Finally, this gives 
\beqs \liminf\limits_{n_2 \ra \infty} |P_{n_2}(s)|^{\frac{1}{n_2^2}} \geq \limsup\limits_{n_1 \ra \infty}|P_{n_1}(s)|^{\frac{1}{n_1^2}}, \eeqs implying 
\beqs 1 \leq  \liminf\limits_{n \ra \infty} |P_n(s)|^{\frac{1}{n^2}} = \lim\limits_{n \ra \infty} |P_n(s)|^{\frac{1}{n^2}} = \limsup\limits_{n \ra \infty} |P_n(s)|^{\frac{1}{n^2}} \leq 2e.\eeqs
\end{proof}
Together with the concavity of $f_n:= |P_n(s)|^{\frac{1}{n^2-1}}$, this implies the following.
\begin{cor}\lab{cor:lip}
Let $\eps > 0$. For all sufficiently large $n$, for all $s$ and $t$ in $\R_+^3$, \beqs |f_n(s) - f_n(t)| < (2e + \eps)|s - t|.\eeqs
\end{cor}
\begin{proof}
Consider the line through $s$ and $t$. The corollary follows by the concavity of $f_n$  on the intersection of this line with $\R_+^3$, the fact that $f_n$ tends to $0$ on the boundary of $\R_+^3$, and Lemma~\ref{lem:3}.
\end{proof}
\begin{cor}
The pointwise limit of the functions $f_n$ is a function $f$ that is  $2e$ Lipschitz and concave.
\end{cor}
\begin{proof}
This follows from Corollary~\ref{cor:lip} and the pointwise convergence of the $f_n$ to a function $f$.
\end{proof}
\subsection{Surface area of facets of $P_n(s)$}

\begin{lem} \lab{lem:ess} There is a universal constant $C > 1$ such that 
for all sufficiently large $n$, the surface area of a codimension $1$ facet of $P_n(s)$ corresponding to $E_i(\T_n)$ is bounded below by $\left(\frac{s_0}{Cs_2}\right)^{\frac{Cs_i}{s_0}} |P_n(s)|^{1 - \frac{1}{n^2 - 1}}.$
\end{lem}
\begin{proof}
 Let $s$ be rescaled by scalar multiplication so that $|P_n(s)| = 1$. Knowing that $|P_n(s)|^{\frac{1}{n^2-1}}$ exists and and has a limit and lies in $[s_0, 2e s_0]$ , we see that $|P_n(s)|^{1 - \frac{1}{n^2 - 1}} \in [\frac{1}{2es_0}, \frac{1}{s_0}].$ Let $F_i$ denote a codimension $1$ facet corresponding to an edge in $E_i(\T_n)$. 
For all sufficiently small $\eps > 0$, we will find a lower bound on the probability that there exists a point $y \in F_i$ such that $\|y - x\|_{\ell_2} < \eps$, when $x$ is sampled at random from $P_n(s)$.
We identify $V(\T_n)$ with  $\Z/n\Z \times \Z/n\Z$ via the unique $\Z$ module isomorphism that maps $[\omega^i]$ to $(1, 0)$ and $[\omega^i \exp(\frac{\pi\imath}{3})]$ to $(0, 1)$. This causes the edges obtained by translating  $\{(0,0), (1, 0), (1, 1), (0,1)\}$ to belong to $E_i(\T_n)$. We further identify $\Z/n\Z \times \Z/n\Z$  with the subset of $\Z^2$ having coordinates in $(-\frac{n}{2}, \frac{n}{2}]$.  Let $T$ be the set of vertices contained in the line segment $\{(a, b)|(a = b) \,\mathrm{and}\,  (|a| \leq \frac{3s_i}{s_0})\}.$ Let $S$ be the set of all  lattice points (vertices) within the convex set $\{(a, b)|(|a -b| \leq 3) \,\mathrm{and} \, (|a + b| \leq \frac{6s_i}{s_0} + 3)\}$ that do not belong to $T$. Without loss of generality, we assume that $F_i$ corresponds to the constraint  $- x(0,0) + x(1, 0) -  x(1, 1) +  x(0,1) \leq s_i$. Let $conv(X)$ be used to denote the convex hull of $X$ for a set of bounded diameter. Let $U = \{u_{-2}, u_{-1}, u_0\}$ be a set of three adjacent vertices not contained in $S\cup T$, but such that exactly two of these vertices are respectively adjacent to  two distinct vertices in $S$. That such a $U$ exists follows from the presence of long line segments in the boundary of $conv(S\cup T)$. Given $x \in P_n(s)$, we define $x_{lin}:conv(U\cup S\cup T) \ra \R$ to be the unique  affine map from the convex hull of $U\cup S\cup T$ to $\R$ which agrees with the values of $x$ on $U$. The function $x_{lin}$ will serve as a baseline for the measurement of fluctuations.  Let $\Lambda_T$ denote the event that 
$\forall (a, a) \in T,$ 

$$ \left|x((a, a)) - x_{lin}((a, a)) - \min\left( \frac{  \left(|a-\frac{1}{2}| - \frac{1}{2}\right)s_0 -  2 s_i}{2}, 0\right)\right|  \leq  \frac{s_0}{20}.$$ 
Let $\Lambda_S$ be the event that for each vertex $v \in S,$ we have   $$-\frac{s_0}{100} \leq x(v)  - x_{lin}(v)  \leq  \frac{s_0}{100}.$$  Let $x_S$ denote the restriction of $x$ to $S$, and likewise define $x_T$, $x_{S\cup T}$ etc.  Let the  cube in $\R^S$ corresponding  to the event $\Lambda_S$ be denoted $Q_S$. Let the polytope in  $\R^T$ corresponding to the event $\Lambda_T$ be denoted $Q_T$. Note that $Q_T$ implicitly depends on $x_S$, but only through the effect of the one constraint $F_i$.
Let $z_S$ be a point in   $[-\frac{s_0}{100}, \frac{s_0}{100}]^S.$ Due to a double layer of separation between $T$ and $V(\T_n)\setminus S$, conditioned on $x_S$ being equal to  $z_S$, the distribution of $x_T$ is independent of the distribution of $x_{V(\T_n)\setminus S}.$ Also, conditioned on $x_s = z_s$, the distribution of $x_T$ is the uniform distribution on a $|T|$ dimensional truncated cube, of sidelength $\frac{s_0}{10}$, the truncation being due the linear constraint $$\langle x_T, \zeta_S\rangle  \geq x((1, 0)) + x((0, 1))- s_i$$ imposed by $F_i$, where $\zeta_S$ is a  vector in $\R^T$ (taking values $1$ on $\{(0,0), (1, 1)\}$ each and $0$ elsewhere). The euclidean distance of the center of this cube to $F_i$ is less than $\frac{s_0}{50}$, so together with Vaalar's theorem \cite{Vaaler} bounding the volume of a central section of a unit cube  from below by $1$, we see that  conditioned on  $\Lambda_T$ and $\Lambda_S$,  the probability that the distance of $x$ to $F_i$ is less than $\eps$ is at least 
$\eps 2^{-|T|}$ for all sufficiently small $\eps$. It remains for us to obtain a positive lower bound on $\p[\Lambda_S \, \mathrm{and} \, \Lambda_T]$ that is independent of $n$ for sufficiently large $n$.
Note that \beq \p[\Lambda_S \, \mathrm{and} \, \Lambda_T] =  \p[\Lambda_T| \Lambda_S]\p[\Lambda_S].\eeq
Let $\mu_{\Lambda_S}$ denote the conditional probability distribution of $x_S$ (supported on $Q_S$) given $\Lambda_S$.
\beqs \p[\Lambda_T| \Lambda_S] & = & \int \p[x_T \in Q_T| x_S = z_S ]\mu_{\Lambda_S}(dz_S)\\
& \geq & \inf_{z_S \in Q_S}\p[x_T \in Q_T| x_S = z_S ].\eeqs
Let $z_S \in Q_S$. Then, the conditional distribution of $x_T$ given that $x_S = z_S$ is the uniform (with respect to Lebesgue) measure on a polytope that is contained in the set of all vectors in  $\R^T$ which when augmented with $z_S$ are $2s_0$ Lipschitz when viewed as functions on ${S\cup T}.$ The latter polytope has volume at most $(4s_0)^{|T|}.$ Since $Q_T$, for any $z_S$, contains a unit cube of side length $s_0/100$, 
\beq  \p[\Lambda_T| \Lambda_S] & \geq & \inf_{z_S \in Q_S}\p[x_T \in Q_T| x_S = z_S ] \geq 400^{-|T|}.\eeq
Finally, we obtain a lower bound on $\p[\Lambda_S].$ 
We say that a vertex $v \in S$ is reachable from $U$ if there is a sequence of  vertices $u_{-2}, u_{-1}, u_0, v_1, \dots, v_k = v$ such that any $4$ consecutive vertices form an edge in $E(\T_n)$ and $v_0, \dots, v_k \in S$. By our construction of $U$, every vertex in $S$ is reachable from $U$, and the length of the path is at most $2|T| + 10$. Consider the values  of $x - x_{lin}$ on $S$. These values cannot exceed $(2|T| + 10)s_2$. Their mean is $0$. Their joint distribution has a density $g_S$ that is logconcave by Pr\'{e}kopa's Theorem~\ref{thm:prekopa}. The probability that $(x - x_{lin})_S$ lies in a translate of $Q_S$ by $t$ is equal to the value of the convolution of $g_S$ with the indicator $I(Q_S)$ of $Q_S$ at $t$. Multiplying by $\left(\frac{50}{s_0}\right)^{|S|}$ (to have unit $L_1$ norm), it follows that each coordinate in any point of the support of $\left(\frac{50}{s_0}\right)^{|S|} I(Q_S) \ast g$ is bounded above by $(2|T| + 12)s_i$, while the mean of this distribution continues to be $0$. The (differential) entropy of $g$ is bounded above by the sum of the entropies of its one dimensional marginals along coordinate directions, which in turn is bounded above by $\ln\left(2e(2|T| + 11)s_2\right)$ by Lemma~\ref{lem:5}. It follows that the supremum of the density of $\left(\frac{50}{s_0}\right)^{|S|} I(Q_S) \ast g$ is at least $\left(2e(2|T| + 12)s_2\right)^{-|S|}$. It is a theorem of Fradelizi \cite{Fradelizi} that the density at the center of mass of a logconcave density on $\R^{|S|}$ is no less than $e^{- |S|}$ multiplied by the supremum of the density. Applied to  $ I(Q_S) \ast g$, this implies that  $$\p[\Lambda_S] \geq \left(100e^2(2|T| + 11)\left(\frac{s_2}{s_0}\right)\right)^{-|S|}.$$
This shows that there is a universal constant $C > 1$ such that 
for all sufficiently large $n$, the surface area of a codimension $1$ facet of $P_n(s)$ corresponding to $E_i(\T_n)$ is bounded below by $\left(\frac{s_0}{Cs_2}\right)^{\frac{Cs_i}{s_0}} |P_n(s)|^{1 - \frac{1}{n^2 - 1}}.$
\end{proof}
By known results on vector partition functions \cite{Brion}, $P_n(s)$ is a piecewise polynomial function of $s$, and each domain of polynomiality is a closed  cone known as a chamber of the associated vector partition function. For a different perspective, see also Lemma $2$ of \cite{Klartag}. It follows by scaling, that these polynomials are homogenous, of degree $n^2 - 1$. Further in the cone $\min(s_0, s_1, s_2) > 0$, $|P_n|$ is $\C^1$ (\ie continuously differentiable) by Lemma~\ref{lem:ess}. 
 

Let \beq\lab{eq:wn}\frac{1}{n^2} \left(\frac{\partial|P_n(s)|}{\partial s_0}, \frac{\partial|P_n(s)|}{\partial s_1}, \frac{\partial|P_n(s)|}{\partial s_2}\right) =: (w_0^{(n)}, w_1^{(n)}, w_2^{(n)}).\eeq

\begin{lem} \lab{lem:surf_upperbd} Fix $s$ with $0 < s_0 \leq s_1 \leq s_2$ and $\eps > 0$, for all sufficiently large $n$, the surface area of a codimension $1$ facet of $P_n(s)$ corresponding to $E_i(\T_n)$ is bounded above  by $\left(\frac{(2e +\eps) s_0}{s_i}\right)|P_n(s)|^{1 - \frac{1}{n^2 - 1}}.$
\end{lem}
\begin{proof} Note that $$\sum_i \left(1 - \frac{1}{n^2}\right)^{-1} s_i \wn_i = |P_n(s)|,$$ which in turn is bounded above by $(2e +\eps) s_0|P_n(s)|^{1 - \frac{1}{n^2 - 1}}$ for sufficiently large $n$.
It follows for each $i \in \{0, 1, 2\}$, that $w_i^{(n)}$ is bounded above by  $\left(\frac{(2e +\eps) s_0}{s_i}\right)|P_n(s)|^{1 - \frac{1}{n^2 - 1}}.$ This completes the proof of this lemma.
\end{proof}
\subsection{Bounds on the $\ell_p$ norm of a point in $P_n(s)$}

Our strategy will be to cover the set of points in $P_n(s)$ that are far from the origin  by a small number of polytopes, each of which is of small volume.
\begin{lem}\lab{lem:infty}
Suppose that $\eps_0 > 0$ and $2 = s_0 \leq s_1 \leq s_2.$ Let $x \in P_n(s)$ be such that 
$\|x\|_{\ell_\infty} \geq \eps_0 n^2.$ Then, for any $p \in [1, \infty)$,
\beq\|x\|_{\ell_p} \geq \left(\frac{\sqrt{3}\eps_0 n}{8 s_2}\right)^{\frac{2}p}\left(\frac{\eps_0 n^2}{2}\right).\eeq
\end{lem}
\begin{proof}
Let the magnitude of the slope of $x$ on a unit triangle $t $ with vertices $v_i, v_j, v_k$ in $\T_n$ be defined to be $\max(|x(v_i) - x(v_j)|, |x(v_j) - x(v_k)|, |x(v_k) - x(v_i)|)$.
 Choose $v_- \in \T_n$ such that $x(v_-)$ is minimal and $v_+ \in \T_n$ such that $x(v_+)$ is maximal.
 Note that the magnitude  of the slope of a triangle $t$ containing $v_-$ cannot exceed $s_2$ because the discrete hessian of all the rhombi containing $v_-$ are bounded above by $s_2$. It is possible to go from one unit triangle with vertices in $\T_n$ to $v_-$ via a sequence of vertices, every $4$ consecutive vertices of which form a unit rhombus,  such that the total number of rhombi is less than $4n$. For this reason the slope of $x$ at no unit triangle can exceed $4ns_2$ in magnitude. Let $v = v_+$ if $x(v_+) \geq - x(v_-)$ and $v = v_-$ otherwise. Therefore, 
$\|x\|_{\ell_\infty} \geq \eps_0 n^2$ implies that any vertex $\widehat v$  within a lattice distance of 
$\frac{\eps_0 n^2}{8ns_2}$ of $v$ satisfies $\frac{x(\widehat v)}{x(v)} > \frac{1}{2},$ implying that $|x(\widehat{v})| \geq \frac{\eps_0 n^2}{2}.$  The number of vertices within a lattice distance of $\frac{\eps_0 n^2}{8ns_2}$ of $v$ is at least $3\left(\frac{\eps_0 n}{8 s_2}\right)^2$. Therefore, \beq\|x\|_{\ell_p}^p \geq 3\left(\frac{\eps_0 n}{8 s_2}\right)^2\left(\frac{\eps_0 n^2}{2}\right)^p.\eeq This implies the lemma.
\end{proof}
\subsection{Polytopes used in the cover}\lab{sec:cover}
We will map $V(\T_n)$ onto $(\Z/n\Z) \times (\Z/n\Z)$ via the unique $\Z$ module isomorphism that maps $1$ to $(1, 0)$ and $\omega$ to $(0, 1)$. 
Without loss of generality (due linearity under scaling by a positive constant), we will assume in this and succeeding sections that that \beq\lim_{n \ra \infty}|P_n(s)|^{\frac{1}{n^2-1}} = 1.\eeq
Let $\eps_0$ be a fixed positive constant. Suppose  $x \in P_n(s)$ satisfies \beq\|x\|_{\ell_\infty} > \eps_0 n^2.\lab{eq:eps0}\eeq


Given $n_1|n_2$, the natural map from $\Z^2$ to $\Z^2/(n_1 \Z^2) = V(\T_{n_1})$ factors through $\Z^2/(n_2 \Z^2) =V(\T_{n_2})$. We denote the respective resulting maps from $V(\T_{n_2})$ to $V(\T_{n_1}) $ by $\phi_{n_2, n_1}$, from $\Z^2$ to $V(\T_{n_2})$ by $\phi_{0, n_2}$ and from $\Z^2$ to $V(\T_{n_1})$ by $\phi_{0, n_1}$.
Given a set of boundary nodes $\b \subseteq V(\T_n)$, and $ x_\b \in \R^{\b}$, we define $Q_{ \b}(x)$ to be the fiber polytope over $x_\b$, that arises from the projection map $\Pi_{\b}$ of $P_n(s)$ onto $\R^{\b}.$ Note that $Q_{\b}(x)$ implicitly depends on $s$. 

Given positive $\eps_0,  \dots, \eps_k$  we will denote by $\eps_{k+1}$, a positive constant whose value may depend on the preceding $\eps_i$  but not on any $\eps_r$ for $r>k$. 
We will associate with $x$, a polytope $Q_n(\eps_1, s, x)$ containing $x$.   Let $o \in V(\T_n)$ be an offset that we will use to define $\b$. 

The polytope $Q_n(\eps_1, s, x)$ is defined as follows. Let $n_2$ be the largest multiple of $\lfloor \eps_1^{-1} \rfloor + 1$ by an odd number, such that the product is less or equal to $n$. Note that $n_2 + 2\lfloor \eps_1^{-1} \rfloor + 1 \geq n$. Let $$n_1 = \frac{n_2}{\lfloor \eps_1^{-1} \rfloor + 1}.$$ We note that by design, $n_1$ is odd.

We define the set $\b_{1} \subseteq V(\T_{n_1})$ of ``boundary vertices" to be all vertices that are either of the form $(0, y)$ or $(1, y)$ or $(x, 0)$ or $(x, 1)$, where $x, y$ range over all of $\Z/(n_1 \Z)$. We define the set $\b_{2} \subseteq V(\T_{n_2})$ to be $\phi_{n_2, n_1}^{-1}(\b_{1}).$

Let $\rho_0:V(\T_{n_2}) \ra \{0, \dots, n_2-1\}^2\subseteq \Z^2$ be the unique map with this range that satisfies $\phi_{0, n_2} \circ \rho_0 = id$ on $V(\T_{n_2})$. We embed $V(\T_{n_2})$ into $V(\T_{n})$ via  
  $\phi_{0, n} \circ \rho_0,$ and define $$\tilde {\b} := \left(\phi_{0, n_3}\circ \rho_0(\b_2)\right)\cup \left(V(\T_{n})\setminus (\phi_{0, n_3}(\{0, \dots, n_2-1\}^2))\right).$$  In other words, $\tilde{\b}$ is the union of the image of $\b_2$ under $\phi_{0, n_3}\circ \rho_0,$ with the set of vertices that do not belong to the range of $\phi_{0, n_3} (\{0, \dots, n_2-1\}^2)$.
Finally we define $\b$ to be $\tilde{\b} + o$, \ie a translation of $\tilde{\b}$ by the offset $o$.
Given $\b$, define $(x_\b)_{quant}$ to be the closest point to $x_\b$, every coordinate of which is an  integer multiple of $\frac{1}{2n^6}$. 

\begin{defn}\lab{def:6.1} We define the polytope $\tilde{Q}_n(\b, s, x)$ as the preimage of $(x_\b)_{quant} + [-\frac{1}{2n^6}, \frac{1}{2n^6}]^\b$ under the coordinate projection $\Pi_\b$ of $P_n(s)$ onto $\R^\b$. 
\end{defn}
Finally, let $o_{min} = o_{min}(\tilde{b}, s, x)$ be a value of the offset $o$ for which the volume of $\tilde{Q}_n(\tilde{\b}+o, s, x)$ achieves its minimum as $o$ ranges over $V(\T_n)$. We define $$Q_n(\eps_1, s, x) := \tilde{Q}_n(\tilde{\b} + o_{min}, s, x).$$
\begin{lem}\lab{lem:polytope_number}
Let $\eps_1 > 0$. Then,  for sufficiently large $n$, the total number of distinct polytopes $Q_n(\eps_1, s, x)$ as $x$ ranges over all points in $P_n(s)$ is at most $ n^{9(8 \eps_1^{-1})n+2}.$
\end{lem}
\begin{proof}
The number of vertices in $\b$ is bounded above by $8 \eps_1^{-1}n$. Also, $x \in P_n(s)$ implies that $\|x\|_{\ell_\infty} < Cn^2$. The number of distinct points of the form $(x_\b)_{quant}$ can therefore be bounded above by $n^{9(8 \eps_1^{-1})n}$ when $n$ is sufficiently large. Since the number of possible offsets is $n^2$, this places an upper bound of $n^{9(8 \eps_1^{-1})n +2}$ on the number of possible polytopes $Q_n(\eps_1, s, x)$.
\end{proof}

\section{Upper bounds on the volumes of covering polytopes}
In this section, $s$ and $x$ and $\eps_1$ will be fixed, so the dependence of  various parameters on them will be suppressed.
For $1 \leq i, j \leq \frac{n_2 }{n_1}$, and offset $o$, we define the $(i,j)^{th}$ square \beq\square_{ij}^o :=o + \phi_{0, n}\left(\left( \left[\frac{(i-1)n_2}{n_1} + 1, \frac{in_2}{n_1}\right]\times \left[\frac{(j-1)n_2}{n_1} + 1,  \frac{jn_2}{n_1}\right]\right)\cap \Z^2\right).\eeq
We also define 
\beq\square^o :=o + \phi_{0, n}\left(\left( \left[ 1, n_2\right]\times \left[1,  n_2\right]\right)\cap \Z^2\right).\eeq
We note that the boundary vertices of each square 
$\square_{ij}^o$ are contained in $\b$. 
Let $\La_{ij}^o$ denote the orthogonal projection of $\R^{V(\T_n)}$ onto the subspace \beq A_{ij}^o := \left\{y\in \R^{\square_{ij}^o}\big|\sum_{k \in \square_{ij}^o}  y_k = 0\right\}.\eeq For any $z \in \tilde{Q}_n(\tilde{\b} + o) - x,$
the euclidean distance between $z$ and this subspace is less than $Cn^3$ by virtue of the upper bound of $Cn^2$ on the Lipschitz constant of $z$ and $x$. For sufficiently large $n$, we eliminate the $C$ and bound this euclidean distance from above by $n^4$. Therefore, for any fixed $o$, 
\beq n^{- \frac{4n_2^2}{n_1^2}}  \left|(\tilde{Q}_n(\tilde{\b} + o) - x)\right| & \leq &  \left|\prod\limits_{{\substack{i/n_1 \in\Z \cap [1, \frac{n_2}{n_1}]\\j/n_1 \in\Z \cap [1, \frac{n_2}{n_1}]}}} \La_{ij}^o (\tilde{Q}_n(\tilde{\b} + o) - x)\right| \\
& = &  \prod\limits_{{\substack{i/n_1 \in\Z \cap [1, \frac{n_2}{n_1}]\\j/n_1 \in\Z \cap [1, \frac{n_2}{n_1}]}}} \left| \La_{ij}^o (\tilde{Q}_n(\tilde{\b} + o) - x)\right|. \eeq

\subsection{Choice of $\tilde t$.}\lab{ssec:uaub}
Let $\C$ denote the open cone in $\R_+^3$ consisting of points $\tilde{u}=(\tilde{u}_0, \tilde{u}_1, \tilde{u}_2)$ such that  
$$\min_\sigma\left(\tilde{u}_{\sigma(0)} + \tilde{u}_{\sigma(1)}- \tilde{u}_{\sigma(2)}\right) > 0,$$ 
where $\sigma$ ranges over all permutations of $\{0, 1, 2\}$.

\begin{lem}

Suppose $0 < e_0 = e_1 \leq e_2.$ then denoting $(e_0, e_1, e_2)$ by $e$, we have $(\wn_0(e), \wn_1(e), \wn_2(e)) \in \C.$

\end{lem}

\begin{proof}
By the anisotropic isoperimetric inequality (\ref{eq:2.2}), applied to $K = P_n(k)$ and $E = P_n(e)$, 
we have 
\beq S_K(E) S_E(K) \geq (n^2 - 1)^2|K| |E|.\eeq Let $k = (2, 2, 2)$.

Then,
\beqs \left(\frac{(n^2 - 1)|K|(e_0 + e_1 + e_2)}{3}\right) \sum_i \wn_i(e) \geq (n^2 - 1)|K| \sum_i \wn_i(e) e_i.\eeqs

This implies that \beq  \frac{\wn_0(e)+ \wn_1(e) + \wn_2(e)}{3} \geq \frac{ \wn_0(e) e_0 + \wn_1(e) e_1 + \wn_2(e) e_2}{e_0 + e_1 + e_2}.\lab{eq:4.2}\eeq
Observe that, $e_0 = e_1 \leq e_2$ and so by symmetry, $\wn_0(e) = \wn_1(e)$.
Thus, (\ref{eq:4.2}) implies that $\wn_2(e) \leq \wn_1(e) = \wn_0(e).$ Putting this together with Lemma~\ref{lem:ess} shows that $(\wn_0(e), \wn_1(e), \wn_2(e)) \in \C.$
\end{proof}

We will handle the following cases:
\ben
\item[(I)] $s$ is in the closure of the set of all points $\tilde{t}$ such that $f$ is once differentiable at $\tilde{t}$ and  $\nabla f(\tilde{t}) \in  \C$.
\item[(II)] $s = (s_0, s_0, s_2),$ for some $s_0, s_2 \in \R_+,$ where $s_0 \leq s_2$.
\een
Note that by Alexandrov's theorem, (see \cite{Colesanti}) the set of all $s \in \R_+^3$ where $f$ is not twice differentiable is of measure $0$. 
\subsubsection{Case (I)}
Let $f_n(s) = |P_n(s)|^{\frac{1}{n^2 - 1}}$ and $f = \lim_{n\ra \infty} f_n.$ Let $w(\tilde{t}) = \nabla f(\tilde t)$ be the gradient of $f$ at a point $\tilde{t}$, assuming that $f$ is differentiable at $\tilde{t}$. We will call a point  $\tilde t$ a surrogate of $s$ if $f$ is once differentiable at $\tilde{t}$ and the following hold. The dot product between $w(\tilde t)$ and $s$ satisfies  \beq 0 \leq w(\tilde t)\cdot s - 1 \leq \eps_2,\lab{eq:dot}\eeq and \beq f(\tilde{t}) = 1.\lab{eq:7.5new}\eeq 

Let us now rescale $\tilde{u}_{a} := \tilde t$ by multiplying it by a suitable positive scalar $\la$ to get $u_{a} = t$ such that $f_{n_1}(u_{a}) = 1.$ 
\subsubsection{Case (II)}
Suppose that the restriction of $f$ to the convex set $$\bar U= \{(u_0, u_0, u_2)| \R_+ \ni u_2 \geq u_0 \in \R_+\}$$ is differentiable at $\tilde{t}$. Suppose also that 
the dot product between $w(\tilde t)$ and $s$ satisfies  \beq 0 \leq w(\tilde t)\cdot s - 1 \leq \eps_2,\lab{eq:dotb}\eeq and \beq f(\tilde{t}) = 1.\lab{eq:7.5newb}\eeq 
Recalling (\ref{eq:wn}), for any $\hat{n} \in \Z_+$,
let \beqs \frac{1}{\hat{n}^2} \left(\frac{\partial|P_{\hat{n}}(t)|}{\partial t_0}, \frac{\partial|P_{\hat{n}}(t)|}{\partial t_1}, \frac{\partial|P_{\hat{n}}(t)|}{\partial t_2}\right) =: (w_0^{(\hat{n})}(t), w_1^{(\hat{n})}(t), w_2^{(\hat{n})}(t)) = w^{(\hat{n})}(t).\eeqs
By the concavity of the $f_n$ and the pointwise convergence of $f_n$ to $f$, the differentiability of $f$ restricted to $\bar U$ at $\tilde t$ together with the symmetry $\wn_0(\tilde t) = \wn_1(\tilde t)$, we see that 
$\lim_{n \ra \infty} \wn(\tilde t)  = w(\tilde t).$ 

 We now state Minkowski's theorem \cite{Klain} for polytopes  and explain that it is applicable in our context.
\begin{thm}\lab{thm:Mink}
Suppose $\e_1, \e_2, \dots, \e_k$ are unit
vectors that do not all lie in a hyperplane of positive codimension, and suppose that $\a_1, \a_2, \dots, \a_k >0.$ If
$\sum_i \a_i \e_i = 0$
then there exists a polytope $P_n$ having facet unit normals $\e_1, \e_2, \dots, \e_k$ and
corresponding facet areas $\a_1, \dots, \a_k$. This polytope is unique up to translation.
\end{thm}
In order to check that these conditions are satisfied by the facet normals of the $P_n(s)$, it suffices to consider the case where $s = 0$ and show that $P_n(0)$ contains no vectorspace of positive dimension. This is indeed the case as we now see. Suppose $x \in P_n(0).$ Then $x$ has mean $0$. The constraints enforce that locally as a function from $V(\T_n)$ to $\R$ its slope is constant. Being a function from the discrete torus to the reals, this slope must be zero.

\subsection{Bounding  $S_{K_{ij}^o}(L_{ij}^o)$ from above}\lab{ssec:appl_isop}
We recall from (\ref{eq:2.2}) that
the anisotropic surface area of $L$ with respect to $K$, denoted $S_K(L)$, satisfies
\beqs S_K(L) \geq m|K|^{\frac{1}{m}} |L|^{\frac{m-1}{m}}.\eeqs 
For $i, j \in n_1 \Z/(n_2\Z)$, let $P_{n_1}^{ij,o}( t)$ be a copy of $L = P_{n_1}( t)$ in $\R^{\square_{ij}^o}$.
Taking $L_{ij}^o$ to be $ P_{n_1}^{ij,o}(t)$  (note that $|P_{n_1}^{ij,o}(t)| = 1$), $K_{ij}^o$ to be $\La_{ij}^o (\tilde{Q}_n(\tilde{\b} + o) -x)$,  and $m = n_1^2 -1$, this gives us 
$$m |K_{ij}^o|^{\frac{1}{m}}  \leq {S_{K_{ij}^o}(L_{ij}^o)}.$$ 


Thus, 
\beqs \prod\limits_{{\substack{i/n_1 \in\Z \cap [1, \frac{n_2}{n_1}]\\j/n_1 \in\Z \cap [1, \frac{n_2}{n_1}]}}} \left| \La_{ij}^o (\tilde{Q}_n(\tilde{\b} + o) - x)\right| & \leq & 
 \prod\limits_{{\substack{i/n_1 \in\Z \cap [1, \frac{n_2}{n_1}]\\j/n_1 \in\Z \cap [1, \frac{n_2}{n_1}]}}} \left(  \frac{S_{K_{ij}^o}(L_{ij}^o)}{m}\right)^m.\eeqs

This implies that 
 \beq \left( n^{- \frac{4n_2^2}{n_1^2}}  \min_{o \in V(\T^n)} \left|(\tilde{Q}_n(\tilde{\b} + o) - x)\right|\right) & \leq & 
\min_{o \in V(\T^n)} \prod\limits_{\substack{i/n_1 \in\Z \cap [1,  \frac{n_2}{n_1}]\\j/n_1 \in\Z \cap [1,  \frac{n_2}{n_1}]}} \left(\frac{S_{K_{ij}^o}(L_{ij}^o)}{m}\right)^m.\nonumber\\\lab{eq:7.9}
\eeq

Recall from Subsection~\ref{sec:prelim} that for $a, b, c$ and $d$ the vertices of a lattice rhombus of side $1$ such that   $ a - d = -z\omega^2,$ $b-a = z,$ $c-b = -z \omega^2 ,$ $d-c = -z,$ for some $z \in \{1, \omega, \omega^2\}.$ In the respective cases when $z= 1, \omega$ or $\omega^2$, we define corresponding sets of lattice rhombi of side $1$ to be $E_0(\mathbb L)$, $E_1(\mathbb L)$ or $E_2(\mathbb L)$. This structure is carried over to $\T_n$ by the map $\phi_{0, n}$ defined in the beginning of Subsection~\ref{sec:cover}.
Recall from the beginning of Subsection~\ref{sec:cover} that we have mapped $V(\T_n)$ on to $(\Z/n\Z) \times (\Z/n\Z)$ by mapping $1$ to $(1, 0)$ and $\omega$ to $(0, 1)$ and extending this map to $V(\T_n)$ via a $\Z$ module homomorphism. In particular, this maps $1 + \omega$ to $(1, 1)$.

Let us examine $S_{K_{ij}^o}(L_{ij}^o)$ for a fixed $i, j$ and $o$. Note that $0 \in K_{ij}^o$. Let us identify $\square_{ij}^o$ with $V(\T_{n_1})$ labelled by $[1, n_1]^2\cap \Z^2$ by mapping the south east corner of $\square_{ij}^o$ onto $(1, 1)$. For $r\in\{0, 1, 2\}$ and $1 \leq k, \ell \leq n_1$, let $u^r_{k\ell}:= u^r_{k\ell}(i, j, o)$ denote the unit outward normal to the  facet of $L_{ij}^o$ that corresponds to the edge in $E_r(\T_{n_1})$, whose south east corner is $(k, \ell)$. 


Consider $h_{k\ell}^r = h_{k\ell}^r(i, j, o)$ to be the maximum value of the functional $\a(a) = \langle a, u^r_{k\ell} \rangle$ as $a$ ranges over $K_{ij}^o$.
We see that \beq S_{K_{ij}^o}(L_{ij}^o) = \sum_{r \in \{0, 1, 2\}} w_r^{(n_1)}(t)\left( \sum_{1 \leq k, \ell \leq n_1} h_{k\ell}^r\right).\lab{eq:7.4}\eeq 

Now, for each $r \in \{0, 1, 2\}$, we define a linear map $D_r$ from $\R^{V(\T_{n'})}$ to $\R^{E_r(\T_{n'})},$ where $n'$ will be a positive integer made clear from context.
Let $f \in \R^{V(\T_{n'})}$ and $(v_1, v_2) \in V(\T_{n'})$. We use $e_r(v_1, v_2)$ to refer to an edge in $E_r(\T_n)$ whose south east corner is the vertex $(v_1, v_2)$. Then,
\ben
\item[(0)] $D_0 f(v_1-1, v_2-1) = \nabla^2 f(e_0(v_1-1, v_2-1)) =  -f(v_1, v_2-1) - f(v_1, v_2) + f(v_1-1, v_2-1) + f(v_1 + 1, v_2).$
\item[(1)] $D_1 f(v_1, v_2) =  \nabla^2 f(e_1(v_1, v_2)) = f(v_1+1, v_2) + f(v_1 , v_2 + 1) - f(v_1, v_2) - f(v_1+1, v_2 + 1).$
\item[(2)] $D_2 f(v_1-1, v_2-1) =  \nabla^2 f(e_2(v_1-1, v_2-1)) = -f(v_1, v_2) - f(v_1-1, v_2) + f(v_1, v_2+1) + f(v_1-1, v_2-1).$
\een
Recall that $K_{ij}^o$ is $\La_{ij}^o (\tilde{Q}_n(\tilde{\b} + o) -x)$. For  linear maps $D_0, D_1$ and $D_2$ described above,  taking $n' = n_1$
we have for $1 \leq k, \ell \leq n_1-2$, and $r \in \{0, 1, 2\}$,
 \beq 0 \leq h_{k\ell}^r = s_r - D_rx(o_1 + i+k, o_2 + j + \ell).\lab{eq:7.5}\eeq
When either $k$ or $\ell$ is one of the numbers $n_1 -1$ or $n_1$, we see that  
$h_{k\ell}^r$ can be larger due to the possibility of the constraints wrapping around. However, it is always true due to the quantization in Definition~\ref{def:6.1}, that
\beq  0 \leq h_{k\ell}^r \leq 2n^{-6} + s_r- D_rx(o_1 + i+k, o_2 + j + \ell).\lab{eq:7.6}\eeq

Let $\Delta = \Delta_t$ be the function from $V(\T_n)$ to $\R$, uniquely specified by the following condition. For any $f:V(\T_n) \ra \R$, and $(v_1, v_2) = v \in V(\T_n)$,
\beq 2(\Delta \ast f)(v) & = & w_0^{(n_1)}(t)(D_0f(v_1 -1, v_2 -1) + D_0f(v_1 -1, v_2))\nonumber\\
& + & w_1^{(n_1)}(t)(D_1f(v_1, v_2 ) + D_1f(v_1 -1, v_2-1))\nonumber\\
& + & w_2^{(n_1)}(t)(D_2f(v_1 -1, v_2 -1) + D_2f(v_1, v_2-1)).\lab{eq:7.7}\eeq
Note that $\Delta$ can be viewed as a self adjoint operator acting on $\mathbb{C}^{V(\T_n)}$ equipped with the standard inner product, but we will find it convenient to define it as a function from $V(\T_n)$ to $\R$ that acts via convolution on complex valued functions defined on $V(\T_n)$.
Let $\Phi$ be the function from $V(\T_n)$ to $\R$, given by \beq \Phi := \frac{\II(\square_{11}^0 - (\frac{n_1 + 1}{2}, \frac{n_1 + 1}{2}))}{m}, \eeq where for a subset $S$ of $V(\T_n)$, $\II(S)$ is the indicator function of $S$, and $\square_{11}^0 - (\frac{n_1 + 1}{2}, \frac{n_1 + 1}{2})$ denotes a centered copy of $\square_{11}^0$ that is symmetric around the origin, whose existence is made possible by the fact that $n_1$ is odd. Again, $\Phi$ can be viewed as a self adjoint operator acting on $\mathbb{C}^{V(\T_n)}$ equipped with the standard inner product.
\begin{lem}\lab{lem:15} For sufficiently large $n$,
\beqs \sum_{o}\sum\limits_{\substack{i/n_1 \in\Z \cap [0, \frac{n_2}{n_1}-1]\\j/n_1 \in\Z \cap [0, \frac{n_2}{n_1}-1]}}\left(\frac{S_{K_{ij}^o}(L_{ij}^o)}{m}\right) & \leq & \frac{n^2 n_2^2}{n_1^2}\left(1 +  {2\eps_2}\right).\eeqs
\end{lem}
\begin{proof}
From (\ref{eq:7.4}), (\ref{eq:7.5}), (\ref{eq:7.6}) and (\ref{eq:7.7}), we observe that \beq \sum_{i, j, o}\left(\frac{S_{K_{ij}^o}(L_{ij}^o)}{m}\right) & = &  \sum_{i, j, o}\sum_{r \in \{0, 1, 2\}} \frac{w_r^{(n_1)}(t)}{m}\left( \sum_{1 \leq k, \ell \leq n_1} h_{k\ell}^r\right)\nonumber\\
& \leq &   \sum_{i, j, o} \frac{2 n^{-6}n_1^2\sum_{r \in \{0, 1, 2\}} w_r^{(n_1)}(t)}{m}\nonumber\\ & + &  \sum_{i, j, o} \frac{\sum_{r \in \{0, 1, 2\}} n_1^2 s_r w_r^{(n_1)}(t)}{m}\nonumber\\ &-&   \sum_{i, j, o} \left(\Phi \ast \Delta \ast x\right)\left(o_1 + i + \frac{n_1 + 1}{2}, o_2 + j + \frac{n_1 + 1}{2}\right).\nonumber\\\lab{eq:le}
\eeq
Note that $$\sum_{i, j, o} \frac{\sum_{r \in \{0, 1, 2\}} n_1^2 t_r w_r^{(n_1)}(t)}{m} =  \frac{n^2 n_2^2}{n_1^2},$$ and that $x$ has mean $0$ and so \beqs  \sum_{i, j, o, k, \ell} \left(\Phi \ast \Delta \ast x\right)(o_1 + i + k, o_2 + j + \ell) = 0.\eeqs
Thus, the expression in (\ref{eq:le}) can be bounded above, for sufficiently large $n$, using volumetric considerations and (\ref{eq:dot}), by 
\beqs n^{-3} +  \frac{n^2 n_2^2}{n_1^2}(1 + \sum_r (w_r^{(n_1)}(t)(s_r - t_r))) & \leq &  n^{-3} \\ & + & \frac{n^2 n_2^2}{n_1^2}(1 + 2\eps_2)\\ & \leq & n^{-3} +  \frac{n^2 n_2^2}{n_1^2}(1 + 3\eps_2).\eeqs
\end{proof}

For a real number $\a$, let  $|\a|_+$ denote $\max(\a, 0)$.
\begin{lem}\lab{lem:7.5} For sufficiently large $n$, 
\beq  \left(\frac{n_1^2}{n^2n_2^2}\right)\sum_{ o}\sum\limits_{\substack{i/n_1 \in\Z \cap [0, \frac{n_2}{n_1}-1]\\j/n_1 \in\Z \cap [0, \frac{n_2}{n_1}-1]}}\Bigg|\left(\frac{S_{K_{ij}^o}(L_{ij}^o)}{m}\right)  -    1  +  \left(\Phi \ast \Delta \ast x\right)\left(o_1 + i + \frac{n_1 + 1}{2}, o_2 + j + \frac{n_1 + 1}{2}\right)\Bigg|_+\nonumber\\\lab{eq:le1}
\eeq is bounded above by $$ (w(\tilde{t}) \cdot (s - \tilde{t})) + \frac{1}{n_1}\left(1 +  3\eps_2\right).$$
\end{lem}
\begin{proof}
Note that by (\ref{eq:dot}), $$ \frac{\sum_{r \in \{0, 1, 2\}} n_1^2 s w_r^{(n_1)}(t)}{m} \leq 1 + \eps_2.$$
We see that  \beq \left(\frac{S_{K_{ij}^o}(L_{ij}^o)}{m}\right) - 1 = \sum_{r \in \{0, 1, 2\}} \frac{w_r^{(n_1)}(t)}{m}\left( \sum_{1 \leq k, \ell \leq n_1}( - t_r +  h_{k\ell}^r)  \right)\lab{eq:7.11}.\eeq

We see that $ (\Phi \ast \Delta \ast x)(v) $
\beq & = & \sum\limits_{-(n_1 - 1)/2 \leq k, \ell \leq (n_1 - 1)/2}\frac{w_0^{(n_1)}(t)}{2m}\big(D_0x(v_1 + k -1, v_2 +\ell -1)  +    D_0x(v_1 + k -1, v_2  + \ell))\nonumber\\
 & + & \sum\limits_{-(n_1 - 1)/2 \leq k, \ell \leq (n_1 - 1)/2}     \frac{w_1^{(n_1)}(t)}{2m}(D_1x(v_1 + k, v_2 + \ell)  +   D_1x(v_1 + k -1, v_2 + \ell -1))\nonumber\\
 & + &    \sum\limits_{-(n_1 - 1)/2 \leq k, \ell \leq (n_1 - 1)/2} \frac{w_2^{(n_1)}(t)}{2m}(D_2x(v_1 + k -1, v_2 + \ell -1)  +   D_2x(v_1 + k, v_2 + \ell -1)\big)\nonumber.\lab{eq:7.7-1}\eeq
We will examine the above expression term by term when $v = o + (i, j) + ((n_1 + 1)/2, (n_1 + 1)/2).$ 

  \beqs & & \sum\limits_{-(n_1 - 1)/2 \leq k, \ell \leq (n_1 - 1)/2}  \frac{w_0^{(n_1)}(t)}{2m}\big(D_0x(v_1 + k -1, v_2 +\ell -1)  +    D_0x(v_1 + k -1, v_2  + \ell))  \\   
&=&  \sum\limits_{\substack{-(n_1 - 1)/2 \leq k \leq (n_1 - 1)/2\\-(n_1 - 1)/2 \leq \ell \leq (n_1 - 1)/2}}\frac{w_0^{(n_1)}(t)}{m}\big(D_0x(v_1 + k -1, v_2 +\ell -1)\big) \\
& + &  \sum\limits_{\substack{-(n_1 - 1)/2 \leq k \leq (n_1 - 1)/2\\\ell \in \{(n_1 + 1)/2\}}}\frac{w_0^{(n_1)}(t)}{2m}\big(D_0x(v_1 + k -1, v_2 +\ell -1)\big)\\
& - &  \sum\limits_{\substack{-(n_1 - 1)/2 \leq k \leq (n_1 - 1)/2\\\ell \in \{-(n_1 - 1)/2\}}}\frac{w_0^{(n_1)}(t)}{2m}\big(D_0x(v_1 + k -1, v_2 +\ell -1)\big).
\eeqs
The above expression is less or equal to
\beqs 
 &&\sum\limits_{\substack{-(n_1 - 1)/2 \leq k \leq (n_1 - 1)/2\\-(n_1 - 1)/2 \leq \ell \leq (n_1 - 1)/2}}\frac{w_0^{(n_1)}(t)}{m}\big(2 n^{-6} + s_0 - h_{k\ell}^0\big) \\
& + &  \sum\limits_{\substack{-(n_1 - 1)/2 \leq k \leq (n_1 - 1)/2\\\ell \in \{(n_1 + 1)/2\}}}\frac{w_0^{(n_1)}(t)}{2m}\big(2 n^{-6} + s_0 - h_{k\ell}^0\big)\\
& - &  \sum\limits_{\substack{-(n_1 - 1)/2 \leq k \leq (n_1 - 1)/2\\\ell \in \{-(n_1 - 1)/2\}}}\frac{w_0^{(n_1)}(t)}{2m}\big( s_0 - h_{k\ell}^0\big).\eeqs
This equals 
\beqs 
 &&\sum\limits_{\substack{-(n_1 - 1)/2 \leq k \leq (n_1 - 1)/2\\-(n_1 - 1)/2 \leq \ell \leq (n_1 - 1)/2}}\frac{w_0^{(n_1)}(t)}{m}\big(2 n^{-6}  + s_0 - h_{k\ell}^0\big) \\
& + &  \sum\limits_{\substack{-(n_1 - 1)/2 \leq k \leq (n_1 - 1)/2\\\ell \in \{(n_1 + 1)/2\}}}\frac{w_0^{(n_1)}(t)}{2m}\big(2 n^{-6}  - h_{k\ell}^0\big)\\
& - &  \sum\limits_{\substack{-(n_1 - 1)/2 \leq k \leq (n_1 - 1)/2\\\ell \in \{-(n_1 - 1)/2\}}}\frac{w_0^{(n_1)}(t)}{2m}\big( - h_{k\ell}^0\big).\eeqs
Adding this to \beqs \sum_{r \in \{0\}} \frac{w_r^{(n_1)}(t)}{m}\left( \sum_{1 \leq k, \ell \leq n_1}( - t_r +  h_{k\ell}^r)  \right),\eeqs from (\ref{eq:7.11}), we get 
\beqs
 &&\sum\limits_{\substack{-(n_1 - 1)/2 \leq k \leq (n_1 - 1)/2\\-(n_1 - 1)/2 \leq \ell \leq (n_1 - 1)/2}}\frac{w_0^{(n_1)}(t)}{m}\big(2 n^{-6} + s_0 - t_0\big) \\
& + &  \sum\limits_{\substack{-(n_1 - 1)/2 \leq k \leq (n_1 - 1)/2\\\ell \in \{(n_1 + 1)/2\}}}\frac{w_0^{(n_1)}(t)}{2m}\big( 2 n^{-6} - h_{k\ell}^0\big)\\
& - &  \sum\limits_{\substack{-(n_1 - 1)/2 \leq k \leq (n_1 - 1)/2\\\ell \in \{-(n_1 - 1)/2\}}}\frac{w_0^{(n_1)}(t)}{2m}\big(  - h_{k\ell}^0\big).\eeqs
This is bounded above by 
\beq
 &&\sum\limits_{\substack{-(n_1 - 1)/2 \leq k \leq (n_1 - 1)/2\\-(n_1 - 1)/2 \leq \ell \leq (n_1 - 1)/2}}\left(\frac{(2 n^{-6} + s_0 - t_0) w_0^{(n_1)}(t)}{m}\right)\\
& + &  \sum\limits_{\substack{-(n_1 - 1)/2 \leq k \leq (n_1 - 1)/2\\\ell \in \{-(n_1 - 1)/2\}}}\frac{w_0^{(n_1)}(t)}{2m}\big(2n^{-6}  + h_{k\ell}^0\big)\\
& \lesssim &  \sum\limits_{\substack{-(n_1 - 1)/2 \leq k \leq (n_1 - 1)/2\\\ell \in \{-(n_1 - 1)/2\}}}\frac{w_0^{(n_1)}(t)(2n_1(s_0 - t_0) + h_{k\ell}^0)}{2m}.\lab{eq:r0}\eeq
Similar calculations done for $r = 1$ give us the following.
\beqs 
& &\sum\limits_{-(n_1 - 1)/2 \leq k, \ell \leq (n_1 - 1)/2}     \frac{w_1^{(n_1)}(t)}{2m}(D_1x(v_1 + k, v_2 + \ell)  +   D_1x(v_1 + k -1, v_2 + \ell -1))\nonumber\\
& + &\sum_{r \in \{1\}} \frac{w_r^{(n_1)}(t)}{m}\left( \sum_{1 \leq k, \ell \leq n_1}( - t_r +  h_{k\ell}^r)  \right)\nonumber\\
& \leq & \sum\limits_{-(n_1 - 1)/2 \leq k, \ell \leq (n_1 - 1)/2}     \frac{(2n^{-6} + s_1 - t_1)w_1^{(n_1)}(t)}{m}\nonumber\\
& + & \sum\limits_{-(n_1 - 1)/2 \leq \ell \leq (n_1 - 1)/2}     \frac{w_1^{(n_1)}(t)}{2m}(D_1x(v_1 -(n_1 + 1)/2, v_2 + \ell) - D_1x(v_1 + (n_1 - 1)/2, v_2 + \ell))\nonumber\\
& + & \sum\limits_{-(n_1 - 1)/2 \leq k \leq (n_1 - 1)/2}     \frac{w_1^{(n_1)}(t)}{2m}(D_1x(v_1 + k, v_2 -(n_1 + 1)/2) - D_1x(v_1 + k, v_2 + (n_1 - 1)/2))\nonumber.
\eeqs 
this is less or equal to 
\beq
&&\sum\limits_{-(n_1 - 1)/2 \leq k, \ell \leq (n_1 - 1)/2}     \frac{(2n^{-6} + s_1 - t_1) w_1^{(n_1)}(t)}{m}\nonumber\\
& + & \sum\limits_{-(n_1 - 1)/2 \leq \ell \leq (n_1 - 1)/2}     \frac{w_1^{(n_1)}(t)}{2m}(2n^{-6} +  h_{(n_1 - 1)/2\,\ell}^1)\nonumber\\
& + & \sum\limits_{-(n_1 - 1)/2 \leq k \leq (n_1 - 1)/2}     \frac{w_1^{(n_1)}(t)}{2m}(2n^{-6} +  h_{k\,(n_1 - 1)/2}^1)\nonumber\\
& \lesssim &\sum\limits_{-(n_1 - 1)/2 \leq \ell \leq (n_1 - 1)/2}     \frac{w_1^{(n_1)}(t)h_{(n_1 - 1)/2\,\ell}^1}{2m}  \nonumber\\
& + & \sum\limits_{-(n_1 - 1)/2 \leq k \leq (n_1 - 1)/2}     \frac{w_1^{(n_1)}(t) (h_{k\,(n_1 - 1)/2}^1 + 2n_1(s_1 - t_1)) }{2m}\lab{eq:r1}
\eeq
Another calculation for $r = 2$, in which the expressions closely resemble the case of $r = 0$, gives us the following.

 \beqs & & \sum\limits_{-(n_1 - 1)/2 \leq k, \ell \leq (n_1 - 1)/2}  \frac{w_2^{(n_1)}(t)}{2m}\big(D_2x(v_1 + k -1, v_2 +\ell -1)  +    D_2x(v_1 + k, v_2  + \ell-1))  \\   
& + &\sum_{r \in \{2\}} \frac{w_r^{(n_1)}(t)}{m}\left( \sum_{1 \leq k, \ell \leq n_1}( - t_r +  h_{k\ell}^r)  \right)\nonumber\\
&=&  \sum\limits_{\substack{-(n_1 - 1)/2 \leq k \leq (n_1 - 1)/2\\-(n_1 - 1)/2 \leq \ell \leq (n_1 - 1)/2}}\frac{(2 n^{-6} + s_2 - t_2)w_2^{(n_1)}(t)}{m} \\
& + &  \sum\limits_{\substack{-(n_1 - 1)/2 \leq \ell \leq (n_1 - 1)/2\\ k \in \{(n_1 + 1)/2\}}}\frac{w_2^{(n_1)}(t)}{2m}\big(D_2x(v_1 + k -1, v_2 +\ell -1)\big)\\
& - &  \sum\limits_{\substack{-(n_1 - 1)/2 \leq \ell \leq (n_1 - 1)/2\\ k \in \{-(n_1 - 1)/2\}}}\frac{w_2^{(n_1)}(t)}{2m}\big(D_2x(v_1 + k -1, v_2 +\ell -1)\big).
\eeqs
This is less or equal to 
 \beq
&& \sum\limits_{\substack{-(n_1 - 1)/2 \leq k \leq (n_1 - 1)/2\\-(n_1 - 1)/2 \leq \ell \leq (n_1 - 1)/2}}\frac{(2 n^{-6}+ s_2 - t_2)w_2^{(n_1)}(t)}{m} \\
& + &  \sum\limits_{\substack{-(n_1 - 1)/2 \leq \ell \leq (n_1 - 1)/2\\ k \in \{(n_1 + 1)/2\}}}\frac{w_2^{(n_1)}(t)}{2m}\big(2n^{-6} - h_{k\ell}^2\big)\\
& - &  \sum\limits_{\substack{-(n_1 - 1)/2 \leq \ell \leq (n_1 - 1)/2\\ k \in \{-(n_1 - 1)/2\}}}\frac{w_2^{(n_1)}(t)}{2m}\big( - h_{k\ell}^2\big)\\
& \lesssim &  \sum\limits_{\substack{-(n_1 - 1)/2 \leq \ell \leq (n_1 - 1)/2\\ k \in \{-(n_1 - 1)/2\}}}\frac{ w_2^{(n_1)}(t) (h_{k\ell}^2 + 2n_1(s_2 - t_2))}{2m}.\lab{eq:r2}
\eeq
From (\ref{eq:r0}), (\ref{eq:r1}) and (\ref{eq:r2}), it follows that 
\beqs \sum_{i, j, o}\Bigg|\left(\frac{S_{K_{ij}^o}(L_{ij}^o)}{m}\right)  -    1  +  \left(\Phi \ast \Delta \ast x\right)\left(o_1 + i + \frac{n_1 + 1}{2}, o_2 + j + \frac{n_1 + 1}{2}\right)\Bigg|_+\eeqs is bounded above by $$\sum_{i,j,o}\left(w^{(n_1)}(t) \cdot (s - t) + \frac{1}{n_1}\left(\frac{S_{K_{ij}^o}(L_{ij}^o)}{m}\right)\right).$$ Using Lemma~\ref{lem:15}, this is in turn bounded above by $$ (w^{(n_1)}(t) \cdot (s - t)) \left(\frac{n^2n_2^2}{n_1^2}\right) + \frac{n^2 n_2^2}{n_1^3}\left(1 +  3\eps_2\right).$$
\end{proof}
\subsection{A lower bound on $\|\Phi\ast\Delta_{t}\ast x\|_{\ell_2}$ when $\|x\|_{\ell_\infty} > \eps_0 n^2.$}
We will first obtain a lower bound on $\|\Phi \ast x\|_{\ell_\infty}$. Let $v$ be a vertex such that $|x(v)| > \eps_0 n^2$. As we observed in the proof of Lemma~\ref{lem:infty}, it is possible to go from one unit triangle with vertices in $\T_n$ to any other via a sequence of vertices, every $4$ consecutive vertices of which form a unit rhombus,  such that the total number of rhombi is less than $4n$. For this reason the slope of $x$ at no unit triangle can exceed $4nt_2$ in magnitude. 
This implies that a sufficient condition for $(\Phi \ast x)(\hat{v}) > \frac{\eps_0n^2}{2}$ is that $n_1 < \frac{\eps_0 n}{32t_2},$ and that the lattice distance between $v$ and $\hat{v}$ is less than  $\frac{\eps_0 n}{32t_2}.$  This is readily achieved, for example, by setting \beq \eps_1 < \frac{\eps_0}{64t_2}.\eeq This implies that \beq\|\Phi \ast x\|_{\ell_p} > \frac{\eps_0 n^2}{2}\left(\frac{\eps_0 n}{32t_2}\right)^\frac{2}{p}.\eeq
We have thus proved the following.
\begin{lem} Assume that $\eps_1 < \frac{\eps_0}{64 t_2}$. Then, for all $p \in [1, \infty]$, $$\|\Phi \ast x\|_{\ell_p} > \frac{\eps_0 n^2}{2}\left(\frac{\eps_0 n}{32t_2}\right)^\frac{2}{p}.$$\lab{lem:kxlp}
\end{lem}
We will now examine $\|\Phi\ast\Delta_{t}\ast x\|_{\ell_2}$, where $t = u_a$.
The convolution of functions over $(\Z/n\Z) \times (\Z/n\Z)$ is a commutative operation, and so 
\beq\Phi\ast \Delta_t \ast x = \Delta_t \ast \Phi \ast x.\eeq
For $(\i, \j) \in (\Z/n\Z)\times(\Z/n\Z)$, and $\omega = \exp(2\pi\imath/n),$ let $\phi_{\i \j}$ be the character of  $(\Z/n\Z)\times(\Z/n\Z)$ given by $\phi_{\i \j}(i, j) := \omega^{\i i + \j j}.$

Let $ y:= \frac{\Phi \ast x}{\|\Phi \ast x\|_{\ell_2}} $ be decomposed in terms of the characters as 
\beq y = \sum_{\i, \j} \theta_{\i \j}\phi_{\i \j},\eeq where, since $y \in \R^{V(\T_n)}$, we have $\theta_{\i \j} = \bar{\theta}_{-\i\, -\j}.$
Writing
\beq \|\Phi\ast\Delta_{t}\ast x\|^2_{\ell_2} & = &  \left(\frac{\|\Delta_{t}\ast\Phi\ast x\|^2_{\ell_2}}{\|\Phi\ast x\|^2_{\ell_2}}\right)\|\Phi\ast x\|^2_{\ell_2},\\
& = & {\|\Delta_{t}\ast y\|^2_{\ell_2}}\|\Phi\ast x\|^2_{\ell_2}.\lab{eq:7.27}\eeq
Let us now study $\|\Delta_{t}\ast y\|^2_{\ell_2}$ by decomposing $y$ using the characters of $(\Z/n\Z)\times(\Z/n\Z)$. Let $\Delta_t \ast\phi_{\i \j} = \la_{\i\j}\phi_{\i\j}.$  We see that 
\beq \Delta_{t}\ast y =  \sum_{\i, \j} \theta_{\i \j}(\la_{\i\j}\phi_{\i \j})\lab{eq:7.28}.\eeq 
In the interest of brevity, we shall drop the allusion to $t$ for the rest of this subsection.
We shall now compute the $\la_{\i\j}$ explicitly in terms of $w^{(n_1)}(t).$
Recall from (\ref{eq:7.7}) that \beqs 2(\Delta \ast y)(v) & = & w_0^{(n_1)}(t)(D_0y(v_1 -1, v_2 -1) + D_0y(v_1 -1, v_2))\nonumber\\
& + & w_1^{(n_1)}(t)(D_1y(v_1, v_2 ) + D_1y(v_1 -1, v_2-1))\nonumber\\
& + & w_2^{(n_1)}(t)(D_2y(v_1 -1, v_2 -1) + D_2y(v_1, v_2-1)).\eeqs
A rearrangement of this gives us 
\beqs 2(\Delta \ast y)(i, j) & = & (- \wone_0 + \wone_1 + \wone_2) (y(i, j+1) - 2 y(i, j) + y(i, j-1))\\
                                            & + & ( \wone_0 - \wone_1 + \wone_2) (y(i+1, j+1) - 2 y(i, j) + y(i-1, j-1))\\
					& + & (\wone_0 + \wone_1 - \wone_2) (y(i+1, j) - 2 y(i, j) + y(i-1, j)).
\eeqs
Let $$\a := - \wone_0 + \wone_1 + \wone_2,$$ $$\beta := \wone_0 - \wone_1 + \wone_2,$$$$\gamma := \wone_0 + \wone_1 - \wone_2.$$
Substituting  $\phi_{\i\j}$ instead of $y$, we see that 
\beqs 2(\Delta \ast \phi_{\i\j})(i, j) & = & \a\omega^{\i i + \j j} (\omega^{\j} - 2  + \omega^{-\j})\\
                                            & + & \beta \omega^{\i i + \j j} (\omega^{\i+\j} - 2  + \omega^{-\i-\j})\\
					& + & \gamma \omega^{\i i + \j j}(\omega^{\i} - 2  + \omega^{-\i}).
\eeqs
Simplifying further, we obtain 
\beqs \Delta \ast \phi_{\i\j} & = & -2\left(\a \sin^2(\frac{\pi \j}{n}) + \beta \sin^2(\frac{- \pi(\i+ \j)}{n}) + \gamma\sin^2(\frac{\pi \i}{n}) \right)\phi_{k\ell}. \eeqs
Thus, \beq\la_{\i\j} =  -2\left(\a \sin^2(\frac{\pi \j}{n}) + \beta \sin^2(\frac{- \pi(\i+ \j)}{n}) + \gamma\sin^2(\frac{\pi \i}{n}) \right).\lab{eq:master}\eeq

Recall from Subsection~\ref{ssec:uaub} that we let $\C$ denote the open cone in $\R_+^3$ consisting of points $\tilde{u}=(\tilde{u}_0, \tilde{u}_1, \tilde{u}_2)$ such that  
$$\min_\sigma\left(\tilde{u}_{\sigma(0)} + \tilde{u}_{\sigma(1)}- \tilde{u}_{\sigma(2)}\right) > 0,$$ 
where $\sigma$ ranges over all permutations of $\{0, 1, 2\}$. 
\subsubsection{Case (I)}
Let $t = u_a$ and $\tilde{t} = \tilde{u}_a$.
We know from Subsection~\ref{ssec:uaub} that $f(\tilde{t}) = 1$ and that $f$ is differentiable at $\tilde{t}$. We already know that $f$ is concave on $\R_+^3.$ For each finite $n$, $\partial f_n (\tilde{t})$ consists of a single point because we have already shown that $f_n$ is $C^1$ in Lemma~\ref{lem:ess}. Further, we see that the solution to the Minkowski problem of finding a polytope with facets having outer normals corresponding to edges in $E(\T_n)$ and surface measures corresponding to $\wone$ has a unique solution from Theorem~\ref{thm:Mink}, and so $f_n$ is strictly concave.
Next, we see that $\wone_0(\tilde{t}) + \wone_1(\tilde{t}) + \wone_2(\tilde{t})\geq f_{n_1}(\tilde{t}) \geq 0.99$ when $n_1$ is sufficiently large because we know from Subsection~\ref{ssec:uaub} that $f(\tilde{t}) = 1$ and that $f$ is  differentiable at $\tilde{t}$. As a result $\lim_{n\ra \infty} \partial f_n(\tilde{t}) = \partial f(\tilde{t}).$
Now, because $\tilde{t} \in \C$, and we know that for $r \in \{0, 1, 2\}$, $$\wone_r(\tilde{t}) > \left(\frac{\tilde{t}_0}{C\tilde{t}_2}\right)^{\frac{C\tilde{t}_r}{\tilde{t}_0}}.$$ from Lemma~\ref{lem:ess}. The second largest among $\a, \beta, \gamma$ must be at least $\min_r \wone_r(\tilde{t}).$ Since we are considering Case (I), and at least two of $k, \ell $ and $k+ \ell$ must be nonzero modulo $n$ (as $x$ is mean $0$), this implies that for any $\i, \j$, $|\la_{\i\j}|$ is bounded below by $\frac{c \min_r \wone_r(\tilde{t})}{n^2}.$ From (\ref{eq:7.27}), (\ref{eq:master}) and the self-adjointness of $\Delta_t$, we see that 
\beqs \|\Phi\ast\Delta_{t}\ast x\|_{\ell_2}  
 & = &   {\|\Delta_{t}\ast y\|_{\ell_2}}\|\Phi\ast x\|_{\ell_2}\\
& \geq & \left(\frac{\tilde{t}_0}{C\tilde{t}_2}\right)^{\frac{C\tilde{t}_2}{\tilde{t}_0}}\frac{\|\Phi\ast x\|_{\ell_2}}{n^2}\\
& \geq & \left(\frac{\tilde{t}_0}{C\tilde{t}_2}\right)^{\frac{C\tilde{t}_2}{\tilde{t}_0}} \left(\frac{\eps_0^2 n}{C\tilde{t}_2}\right)\\
& \geq & \left(\frac{s_0}{Cs_2}\right)^{\frac{Cs_2}{s_0}} \left(\frac{\eps_0^2 n}{Cs_2}\right).\eeqs

Thus,
\beq \|\Phi\ast\Delta_{u_a}\ast x\|_{\ell_2} \geq \left(\frac{s_0}{Cs_2}\right)^{\frac{Cs_2}{s_0}} \left(\frac{\eps_0^2 n}{Cs_2}\right).\lab{eq:end1}\eeq

\subsubsection{Case (II)}

The argument above applies to Case (II) as well and gives us 
\beq \|\Phi\ast\Delta_{t}\ast x\|_{\ell_2} \geq \left(\frac{s_0}{Cs_2}\right)^{\frac{Cs_2}{s_0}} \left(\frac{\eps_0^2 n}{Cs_2}\right).\lab{eq:end1-sept10-2019}\eeq

\subsection{Upper bound on $|\tilde{Q}_n(\tilde{\b} + o, s, x)|$}
\begin{lem}\lab{lem:polytope_volume}
Let $s$ be as in Case (I) or Case (II), and suppose that $x \in P_n(s)$ satisfies $\|x\|_{\ell_\infty} \geq \eps_0 n^2.$ Let $\eps_1 = \frac{\eps_0}{100s_2}$.  Then, 
the volume of the polytope $Q_n(\eps_1, s, x)$ can be bounded from above as follows.
 \beqs |Q_n(\eps_1, s, x)| \leq   n^{\frac{4n_2^2}{n_1^2}}  \exp( - (n_1^2 - 1) \eps_0^4/24). \eeqs
\end{lem}
\begin{proof}

We need a suitable bound from below  on \beqs \inf_o  \sum_{\substack{i/n_1 \in\Z \cap [0, \frac{n_2}{n_1}-1]\\j/n_1 \in\Z \cap [0, \frac{n_2}{n_1}-1]}}\left(\Phi \ast \Delta \ast x\right)\left({o}_1 + i + \frac{n_1 + 1}{2}, {o}_2 + j + \frac{n_1 + 1}{2}\right).\eeqs
Such a bound can be obtained by first observing that \beqs \sum_{v \in \square^o} \left(\Delta \ast x\right)\left(v_1 , v_2\right) =  (-1) \sum_{v \in V(\T_n) \setminus \square^o} \left(\Delta \ast x\right)\left(v_1, v_2\right).\eeqs

Recall from (\ref{eq:7.7}) that  for any $f:V(\T_n) \ra \R$, and $(v_1, v_2) = v \in V(\T_n)$,
\beqs 2(\Delta \ast f)(v) & = & w_0^{(n_1)}(t)(D_0f(v_1 -1, v_2 -1) + D_0f(v_1 -1, v_2))\nonumber\\
& + & w_1^{(n_1)}(t)(D_1f(v_1, v_2 ) + D_1f(v_1 -1, v_2-1))\nonumber\\
& + & w_2^{(n_1)}(t)(D_2f(v_1 -1, v_2 -1) + D_2f(v_1, v_2-1)).\eeqs
This expression on the right, is less or equal to $2.02$.

This implies that 
\beqs (-1.01) \sum_{v \in V(\T_n) \setminus \square^o} \left(\Delta \ast x\right)\left(v_1 , v_2\right) & > &  (-1.01)|V(\T_n) \setminus \square^o|\\
& > & (-1.01 n)(4\eps_1^{-1} + 2).\eeqs
Thus, we have proved the following.
\beq \inf_o  \sum\limits_{\substack{i/n_1 \in\Z \cap [0, \frac{n_2}{n_1}-1]\\j/n_1 \in\Z \cap [0, \frac{n_2}{n_1}-1]}} \left(\Phi \ast \Delta \ast x\right)\left({o}_1 + i + \frac{n_1 + 1}{2}, {o}_2 + j + \frac{n_1 + 1}{2}\right) &  >  & -1.01\left(\frac{n}{m}\right)(4\eps_1^{-1} + 2)\nonumber\\
& > & -1.01 \left(\frac{C}{n}\right) \eps_1^{-3}.\lab{eq:8.1}\eeq

Similarly, from (\ref{eq:7.7}) we observe that for all $o$, 
\beq   \inf_{\substack{i/n_1 \in\Z \cap [0, \frac{n_2}{n_1}-1]\\j/n_1 \in\Z \cap [0, \frac{n_2}{n_1}-1]}}\left(\Phi \ast \Delta \ast x\right)\left({o}_1 +i + \frac{n_1 + 1}{2}, {o}_2 + j + \frac{n_1 + 1}{2}\right) &\geq& \left(\frac{-1.01}{m}\right)|V(\T_n)|,\nonumber\\
& \geq & \left(\frac{- 5}{\eps_1^2}\right),\lab{eq:8.2}\eeq
when $n_1$ is sufficiently large.

Recall from Lemma~\ref{lem:7.5} that
for sufficiently large $n$, 
\beqs  \left(\frac{n_1^2}{n^2n_2^2}\right)\sum_{i, j, o}\Bigg|\left(\frac{S_{K_{ij}^o}(L_{ij}^o)}{m}\right)  -    1  +  \left(\Phi \ast \Delta \ast x\right)\left(o_1 + i + \frac{n_1 + 1}{2}, o_2 + j + \frac{n_1 + 1}{2}\right)\Bigg|_+\nonumber
\eeqs is bounded above by $ (w(\tilde{t}) \cdot (s - \tilde{t})) + \frac{1}{n_1}\left(1 +  3\eps_2\right).$ This is in turn less or equal to $ \eps_2 + \frac{1}{n_1}\left(1 +  3\eps_2\right) < 2\eps_2.$

Pick a $\bar{o}$ uniformly at random. By Lemma~\ref{lem:7.5}, and Markov's inequality, with probability greater or equal to $  1/2$, 
\beqs  \left(\frac{n_1^2}{n_2^2}\right)\sum_{i, j}\Bigg|\left(\frac{S_{K_{ij}^{\bar{o}}}(L_{ij}^{\bar{o}})}{m}\right)  -    1  +  \left(\Phi \ast \Delta \ast x\right)\left(\bar{o}_1 + i + \frac{n_1 + 1}{2}, \bar{o}_2 + j + \frac{n_1 + 1}{2}\right)\Bigg|_+
\eeqs
\beqs < {4\eps_2}.\eeqs
Let \beqs\eta(i, j) := \Bigg|\left(\frac{S_{K_{ij}^{\bar{o}}}(L_{ij}^{\bar{o}})}{m}\right)  -    1  +  \left(\Phi \ast \Delta \ast x\right)\left(\bar{o}_1 + i + \frac{n_1 + 1}{2}, \bar{o}_2 + j + \frac{n_1 + 1}{2}\right)\Bigg|_+.\eeqs
Thus, with probability at least $\frac{1}{2}$,
\beq  \sum_{i, j}\eta(i, j)  < \frac{16\eps_2}{\eps_1^2}.\nonumber\\ \lab{eq:8.3}\eeq

Recall from (\ref{eq:end1}) that 
\beqs \|\Phi\ast\Delta_{u_a}\ast x\|_{\ell_2} \geq \left(\frac{s_0}{Cs_2}\right)^{\frac{Cs_2}{s_0}} \left(\frac{\eps_0^2 n}{Cs_2}\right).\eeqs
Therefore, for $\De$ that equals  $\De_{u_a}$ 
Markov's inequality informs us that with probability greater or equal to $ \frac{2}{3}$, 
\beq  \sum\limits_{\substack{i/n_1 \in\Z \cap [0, \frac{n_2}{n_1}-1]\\j/n_1 \in\Z \cap [0, \frac{n_2}{n_1}-1]}} \left(\Phi \ast \Delta \ast x\right)\left({\bar o}_1 + i + \frac{n_1 + 1}{2}, {\bar o}_2 + j + \frac{n_1 + 1}{2}\right)^2  
& > & \eps_0^4  \eps_1^{-2}.\nonumber \\ \lab{eq:8.4}\eeq
Let \beq \zeta(i, j) := (-1)\left(\Phi \ast \Delta \ast x\right)\left({\bar o}_1 + i + \frac{n_1 + 1}{2}, {\bar o}_2 + j + \frac{n_1 + 1}{2}\right)+ \eta(i, j).\nonumber\\ \lab{eq:defzeta}\eeq
Recall from (\ref{eq:7.9}) that 

 \beqs \left( n^{- \frac{4n_2^2}{n_1^2}}  \min_{o \in V(\T^n)} \left|(\tilde{Q}_n(\tilde{\b} + o) - x)\right|\right) & \leq & 
\min_{o \in V(\T^n)} \prod\limits_{\substack{i/n_1 \in\Z \cap [1,  \frac{n_2}{n_1}]\\j/n_1 \in\Z \cap [1,  \frac{n_2}{n_1}]}} \left(\frac{S_{K_{ij}^o}(L_{ij}^o)}{m}\right)^m.\nonumber\\
\eeqs
We observe that the right hand side above can be bounded above as follows. \beqs 
\min_{o \in V(\T^n)} \prod\limits_{\substack{i/n_1 \in\Z \cap [1,  \frac{n_2}{n_1}]\\j/n_1 \in\Z \cap [1,  \frac{n_2}{n_1}]}} \left(\frac{S_{K_{ij}^o}(L_{ij}^o)}{m}\right)^m
& \leq & \prod\limits_{\substack{i/n_1 \in\Z \cap [1,  \frac{n_2}{n_1}]\\j/n_1 \in\Z \cap [1,  \frac{n_2}{n_1}]}} \left( 1 + \zeta(i, j)\right)^m.
\eeqs
For our purposes, it therefore suffices to get a good upper bound on \beq \prod\limits_{\substack{i/n_1 \in\Z \cap [1,  \frac{n_2}{n_1}]\\j/n_1 \in\Z \cap [1,  \frac{n_2}{n_1}]}} \left( 1 + \zeta(i, j)\right)^m.
\eeq

By the inequality \beq \ln a \leq a-1 - \frac{(a-1)^2}{2\max(a, 1)}\eeq for $a \in (0, \infty),$ we see that 
\beq \ln \left( 1 + \zeta(i, j)\right) \leq \zeta(i, j) - \frac{\zeta(i, j)^2}{2\max(\zeta(i, j), 1)}.\eeq 
By (\ref{eq:8.2}) and (\ref{eq:8.3}), it follows that \beq \max(\zeta(i, j), 1) \leq \frac{5 + 16\eps_2}{\eps_1^2} < \frac{6}{\eps_1^2}.\eeq

Now further using (\ref{eq:8.1}), (\ref{eq:8.3}) and (\ref{eq:8.4}) we have 
\beq \sum\limits_{\substack{i/n_1 \in\Z \cap [1,  \frac{n_2}{n_1}]\\j/n_1 \in\Z \cap [1,  \frac{n_2}{n_1}]}}  \ln \left( 1 + \zeta(i, j)\right) \leq  1.01 \left(\frac{C}{n}\right) \eps_1^{-3} + \frac{16\eps_2}{\eps_1^2} - \frac{\eps_0^4 /\eps_1^2}{\frac{12}{\eps_1^2}}.\eeq
When, $n \ra \infty$, we have $\eps_1 < \frac{\eps_0}{64 t_2}$ (thus fulfiling  the requirements of Lemma~\ref{lem:kxlp}) and suppose that $\eps_2$ to tend to $0$  in such a way, that for all sufficiently large $n$,
\beq - 1.01 \left(\frac{C}{n}\right) \eps_1^{-3} -  \frac{16\eps_2}{\eps_1^2} + \frac{\eps_0^4/\eps_1^2}{\frac{12}{\eps_1^2}} & \geq &  -  \frac{20\eps_2}{\eps_1^2} + \eps_0^4/12\\
& \geq & \eps_0^4/24.\eeq

This gives us 
\beq \prod\limits_{\substack{i/n_1 \in\Z \cap [1,  \frac{n_2}{n_1}]\\j/n_1 \in\Z \cap [1,  \frac{n_2}{n_1}]}}   \left( 1 + \zeta(i, j)\right)^m \leq \exp( - m \eps_0^4/24). \eeq

It now follows from (\ref{eq:7.9}) that 
 \beq |Q_n(\eps_1, s, x)| & = & \min_{o \in V(\T^n)} \left|(\tilde{Q}_n(\tilde{\b} + o) - x)\right| \\
& \leq &  n^{\frac{4n_2^2}{n_1^2}}  \exp( - m \eps_0^4/24). \lab{eq:poly_end}\eeq
\end{proof}
\section{Existence of a scaling limit}
Recall from Subsection~\ref{sec:prelim} that  given  $s = (s_0, s_1, s_2)\in \R_+^3,$  we define $P_n(s)$ to be the bounded polytope of  all functions $x:V(\T_n) \ra \R$ such that $\sum_{v \in V(\T_n)} x(v) = 0$ and $\nabla^2(x)\preccurlyeq s$. 

\begin{lem}\lab{lem:diameter}
The $\ell_\infty$ diameter of $P_n(s)$ is greater than $(s_1 + s_2)\lfloor n/2\rfloor^2/4$ for all $n$ greater than $1$.
\end{lem}
\begin{proof}
Recall from Lemma~\ref{lem:2.3} that there is a unique quadratic function $q$ from $\mathbb L$ to $\R$ such that $\nabla^2q$ satisfies the following.

\ben 
\item  $\nabla^2q(e) = - s_0,$ if $e \in E_0(\mathbb L)$.
\item $\nabla^2q(e)  =  - s_1,$  if $e \in E_1(\mathbb L )$.
\item $\nabla^2q(e)  =  - s_2,$ if  $e \in E_2(\mathbb L)$.
\item $q(0) = q(n) = q(n\omega) = 0$.
\een
We define the function $r$ from $\R^2$ to $\R$ to be the unique function that agrees with $q$ on $n \mathbb L$, but is defined at all points of $\R^2\setminus n{\mathbb L}$ by piecewise linear extension.
In other words, the epigraph of $-r$ is the convex hull of all points of the form $(v, -q(v))$ as $v$ ranges over $n \mathbb L$. 
The function $r - q$ restricted to $\mathbb L$ is invariant under shifts by elements in $n \mathbb L$ and so  can be viewed as a function from $V(\T_n)$ to $\R$. The function from $V(\T_n)$ to $\R$ obtained by adding a suitable constant $\kappa$ to $r - q$ such that it has zero mean is a member of $P_n(s)$. We readily see, by examining one of the sides of a fundamental triangle in $n\mathbb L$ that $\|r- q + \kappa\|_{\ell_\infty}$ is at least $(s_1 + s_2)\lfloor n/2\rfloor^2/4$. Since the constant function taking value $0$ belongs to $P_n(s)$, the lemma follows.
\end{proof}
Recall from Subsection~\ref{ssec:uaub} that we denote by $\C$,  the open cone in $\R_+^3$ consisting of points $\tilde{u}=(\tilde{u}_0, \tilde{u}_1, \tilde{u}_2)$ such that  
$$\min_\sigma\left(\tilde{u}_{\sigma(0)} + \tilde{u}_{\sigma(1)}- \tilde{u}_{\sigma(2)}\right) > 0,$$ 
where $\sigma$ ranges over all permutations of $\{0, 1, 2\}$. 
\begin{thm}\lab{thm:main}
Let $s = (s_0, s_1, s_2) \in \R_+^3$ where $s_0 \leq s_1 \leq s_2$ and either of the following two conditions is true.
\ben
\item[(I)] $s$ is in the closure of the set of all points $\tilde{t}$ such that $f = \lim_{n \ra \infty} |P_n(s)|^{\frac{1}{n^2 - 1}}$ is once differentiable at $\tilde{t}$ and  $\nabla f(\tilde{t}) \in  \C$.
\item[(II)] $s = (s_0, s_0, s_2),$ for some $s_0, s_2 \in \R_+,$ where $s_0 \leq s_2$.
\een
For $\eps_0 > 0$, let $p_n(s, \eps_0)$ denote the probability that a point sampled from the normalized Lebesgue measure on $P_n(s)$ has  $\ell_\infty$ norm greater than $\eps_0n^2$.
Then,
\beqs \lim_{n\ra \infty} p_n(s, \eps_0)  = 0.\eeqs
\end{thm}
\begin{proof} 
Recall from Lemma~\ref{lem:polytope_number} that for sufficiently large $n$, the total number of distinct polytopes $Q_n(\eps_1, s, x)$ as $x$ ranges over all points in $P_n(s)$ is at most $ n^{9(8 \eps_1^{-1})n+2}.$ Together with Lemma~\ref{lem:polytope_volume}, this implies that the volume of the set of all points $x$ in the polytope $P_n(s)$  whose $\ell_\infty$ norm is at least $\eps_0n^2$ is less than 
$$ n^{9(8 \eps_1^{-1})n+2}  n^{\frac{4n_2^2}{n_1^2}}  \exp( - (n_1^2 - 1) \eps_0^4/24) |P_n(s)|. $$ The multiplicative factor $  n^{9(8 \eps_1^{-1})n+2}  n^{\frac{4n_2^2}{n_1^2}}  \exp( - (n_1^2 - 1) \eps_0^4/24)$ tends to $0$ for any fixed $\eps_0 > 0$ as $n \ra \infty$, completing the proof.
\end{proof}

\section{Acknowledgements}
I am grateful to Scott Sheffield for several helpful discussions on the topic of this paper as well as on random surfaces in general. I thank Bo'az Klartag for a helpful discussion on the stability of the Brunn-Minkowski inequality. I was partly supported by NSF
grant DMS-1620102 and a Ramanujan Fellowship.

\end{document}